\documentclass{amsart}
\usepackage{euscript, amssymb, mathrsfs, amsmath, amsthm, float, stmaryrd}
\usepackage{xy, mathrsfs}
\usepackage{enumerate, comment}
\usepackage{tikz-cd}
\usetikzlibrary{backgrounds}
\usepackage{placeins}
\usepackage{mathtools}
\xyoption{all}
\usepackage{hyperref}
\usepackage{cleveref}
\title{Notes on model structures on preorders}
\date{December 2025}
\author{Andrew Salch}
\author{Gunjeet Singh}
\theoremstyle{plain}
\newtheorem{prop}{Proposition}[section]
\newtheorem{theorem}[prop]{Theorem}
\newtheorem{corollary}[prop]{Corollary}
\newtheorem{lemma}[prop]{Lemma}
\newtheorem{definition}[prop]{Definition}
\newtheorem{definition-proposition}[prop]{Definition-Proposition}
\newtheorem{definition-theorem}[prop]{Definition-Theorem}
\newtheorem{notation}[prop]{Notation}

\newcounter{lettered}
\newtheorem{letteredtheorem}[lettered]{Theorem}
\newtheorem{letteredcorollary}[lettered]{Corollary}

\theoremstyle{definition}
\newtheorem{remark}[prop]{Remark}

\newtheorem{question}[prop]{Question}

\newtheorem{observation}[prop]{Observation}

\DeclareMathOperator{\cof}{{\rm cof}}
\DeclareMathOperator{\WE}{{\rm we}}
\DeclareMathOperator{\fib}{{\rm fib}}
\DeclareMathOperator{\possible}{{\diamond}}
\DeclareMathOperator{\Rels}{{\rm Rels}}
\DeclareMathOperator{\op}{{\rm op}}
\DeclareMathOperator{\Top}{{\rm Top}}
\DeclareMathOperator{\Inj}{{\rm Inj}}
\DeclareMathOperator{\Proj}{{\rm Proj}}
\DeclareMathOperator{\Moore}{{\rm Moore}}
\DeclareMathOperator{\Int}{{\rm Int}}
\DeclareMathOperator{\Cl}{{\rm Cl}}
\DeclareMathOperator{\Mor}{{\rm Mor}}
\DeclareMathOperator{\Fib}{{\rm Fib}}
\DeclareMathOperator{\Cofib}{{\rm Cof}}

\begin{document}
\begin{abstract} 
Given subsets $\mathcal{C},\mathcal{F}$ of a preorder $\mathcal{A}$, we give necessary and sufficient conditions for $\mathcal{A}$ to admit the structure of a model category whose cofibrant objects are $\mathcal{C}$ and whose fibrant objects are $\mathcal{F}$. We give various classification results for model structures on preorders by describing model structures in terms of their fibrant and cofibrant objects, or in terms of their (co)fibrant replacment (co)monads. This leads to a construction which takes topologies and matroids as input, and produces model structures on Boolean algebras. We carry out some detailed case studies, calculating all model structures on small Boolean algebras, and all the Bousfield localization and colocalization relations between them.
\end{abstract}
\maketitle

\section{Introduction}

Given a category $\mathcal{C}$, one might like to know something about all the model structures on $\mathcal{C}$, i.e., all the model categories whose underlying category is $\mathcal{C}$. For most categories $\mathcal{C}$, the problem of calculating all model structures on $\mathcal{C}$ seems intractable.

Among all small categories, the preorders\footnote{Here is a review of the most basic ideas. A preorder is a set equipped with a reflexive, transitive binary relation. Put another way: a preorder is what you get when you drop antisymmetry from the definition of a partially-ordered set. By a very old and well-known construction, one can think of a preorder $S$ as a small category by letting the elements of $S$ be the objects of the category, and by letting the set of morphisms $\hom_S(X,Y)$ in the category have a single element if $X\leq Y$, and letting $\hom_S(X,Y)$ be empty otherwise. This yields an embedding of the category of preorders into the category of small categories. The small categories which come from preorders in this way are precisely the small categories in which each hom-set has at most one element.} are particularly easy to study. Consequently it seems reasonable to treat preorders as a class of ``toy examples'' of categories on which all the model structures can perhaps be understood. Perhaps model structures on preorders are equivalent to some algebraic or order-theoretic or combinatorial data of a reasonable simple or familiar kind, and could consequently be classified.

That is what this paper is about. Here are the broad conclusions:
\begin{itemize}
\item A model structure on a preorder satisfying a certain simple and mild condition (i.e., a ``strong'' model structure in a sense to be defined in \cref{Existence of model structures...}) is equivalent to a reasonable piece of order-theoretic data, which amounts to specifying just the set of fibrant objects and the set of cofibrant objects. See Theorem \ref{lettered thm c} and Corollary \ref{lettered cor d}.
\item Those theorems let us construct model structures on Boolean algebras out of topologies and out of matroids; see \cref{Model structures constructed from...}. This is a new and perhaps interesting  construction.
\item If one wants to classify all model structures on a preorder, not just the ``strong'' model structures, then the job is harder. If one specifies a putative set of fibrant objects and a putative set of cofibrant objects, we give necessary and sufficient conditions (Corollary \ref{lettered cor b}) for the preorder to admit a model structure with precisely those (co)fibrant objects. Only one such model structure will be a strong model structure, but there may be many which are non-strong.
\item 
Proposition \ref{fact brackets and fact systems and model structs} establishes that model structures on any preorder are in bijection with ``factorization bracket pairs'' on that preorder. This is not a deep result at all: factorization bracket pairs are an algebraic piece of structure, defined in Definition \ref{def of fact bracket}, but they are merely a straightforward algebraic translation of what a model structure on a preorder must be. 
\item 
It is at least true that the algebraic formalism of factorization bracket pairs lends itself to a systematic calculation of all model structures (not just strong model structures) on a fixed preorder, with computer assistance. We programmed a computer to calculate all model structures on small Boolean algebras. The results are structurally interesting, and we report on them below, in \cref{case studies...} and in \cref{Explicit examples}.
\end{itemize}


\subsection{Existence of model structures with specified (co)fibrant objects}
\label{Existence of model structures...}

Droz and Zakharevich, in \cite{MR4226148}, study model categories whose underlying categories are preorders, and they provide a classification of such model categories {\em up to Quillen equivalence}: two such model categories are Quillen-equivalent if and only if their homotopy categories are equivalent. By contrast, in this paper we study model structures on a fixed preorder, ``on the nose,'' i.e., not up to Quillen equivalence. We prove some results of an explicit computational nature, as well as some results of a more general and conceptual nature. 

Here is the setup for our first conceptual result. Given a model structure on a preorder $\mathcal{A}$, one gets a fibrant replacement monad $F$ on $\mathcal{A}$, and a cofibrant replacement comonad on $C$ on $\mathcal{A}$. One might want to go the other way: starting with a monad $F$ and a comonad $C$ on a preorder $\mathcal{A}$, does there exist a model structure on $\mathcal{A}$ whose fibrant and cofibrant replacement (co)monads are $F$ and $C$, respectively? Here is the relevant theorem:
\begin{letteredtheorem} (Theorem \ref{characterization thm}) \label{lettered thm a} Let $\mathcal{A}$ be a bicomplete preorder, let $F$ be a monad on $\mathcal{A}$, and let $C$ be a comonad on $\mathcal{A}$. Then the following conditions are equivalent:
\begin{itemize}
\item There exists a model structure on $\mathcal{A}$ whose cofibrant replacement comonad is naturally isomorphic to $C$, and whose fibrant replacement monad is naturally isomorphic to $F$.
\item $F$ and $C$ are each idempotent, $F$ commutes with $C$ up to isomorphism, and if $X,Y$ are elements of $\mathcal{A}$ such that $X\leq CY$ and $FX\leq Y$, then $FX\leq CY$.
\end{itemize}
\end{letteredtheorem}
The following corollary of Theorem \ref{lettered thm a} indicates one way that we expect the theorem can be useful: if we specify the desired cofibrant objects and desired fibrant objects in some putative model structure on a bicomplete preorder $\mathcal{A}$, then our result tells us precisely what conditions must be satisfied in order for there to exist a model structure on $\mathcal{A}$ with the specific cofibrant and fibrant objects. 
\begin{letteredcorollary}  \label{lettered cor b} Let $\mathcal{A}$ be a bicomplete preorder, and let $\mathcal{F},\mathcal{C}$ be subsets of $\mathcal{A}$. Then there exists a model structure on $\mathcal{A}$ whose fibrant objects are precisely the members of $\mathcal{F}$, and whose cofibrant objects are precisely the members of $\mathcal{C}$, if and only if all of the following conditions are satisfied:
\begin{enumerate}
\item $\mathcal{F}$ and $\mathcal{C}$ are each closed under isomorphism in $\mathcal{A}$.
\item $\mathcal{F}$ is reflective, i.e., for each element $x\in\mathcal{A}$, the set of upper bounds for $x$ which lie in $\mathcal{F}$ has a least element.
\item $\mathcal{C}$ is coreflective, i.e., for each element $x\in\mathcal{A}$, the set of lower bounds for $x$ which lie in $\mathcal{C}$ has a greatest element.
\item If $x\in\mathcal{A}$ and $f\in\mathcal{F}$ and $c\in\mathcal{C}$ satisfy the conditions\footnote{This mazelike statement is a choice-free, first-order way of saying that a reflector functor $F: \mathcal{A}\rightarrow\mathcal{F}$ commutes with a coreflector functor $C: \mathcal{A}\rightarrow\mathcal{C}$ up to isomorphism in $\mathcal{A}$.}
\begin{itemize}
\item $f$ is minimal among members of $\mathcal{F}$ which are greater than or equal to every lower bound for $x$ in $\mathcal{C}$,
\item and $c$ is maximal among members of $\mathcal{C}$ which is less than or equal to every upper bound for $x$ in $\mathcal{F}$,
\end{itemize}
then $b\cong c$, i.e., $b\leq c\leq b$. 
\item If $x,y\in\mathcal{A}$ are such that $x$ is less than or equal to every lower bound for $y$ in $\mathcal{C}$, and $y$ is greater than or equal to every upper bound for $x$ in $\mathcal{F}$, then some upper bound for $x$ in $\mathcal{F}$ is less than or equal to some lower bound for $y$ in $\mathcal{C}$.\footnote{This statement is a choice-free, first-order rephrasing of ``If $a\leq C(b)$ and $F(a)\leq b$, then $F(a)\leq C(b)$.''}
\end{enumerate}
If such a model structure on $\mathcal{A}$ indeed exists, then its homotopy category is precisely the sub-preorder $\mathcal{F}\cap \mathcal{C}$ of $\mathcal{A}$.
\end{letteredcorollary}

Suppose we are given a bicomplete category $\mathcal{A}$ and, speaking informally, some ``piece'' of a model structure: perhaps we are given a putative class of weak equivalences, or a putative class of cofibrant objects, etc., and we ask whether there indeed exists a model structure on $\mathcal{A}$ with the specified piece of structure. One says (following Droz and Zakharevich \cite{MR4226148}) that the given piece of structure {\em extends to a model structure.} The main theorem of \cite{MR4226148} is that there is no necessary and sufficient condition {\em expressible in first-order logic}\footnote{Here by ``expressible in first-order logic,'' we (and Droz--Zakharevich) mean ``expressible using quantification over elements (which may themselves be objects and/or morphisms in a category), but without quantifying over sets or functions.''} for whether a putative class of weak equivalences $W$ extends to a model structure on $\mathcal{A}$. In fact their proof is really concerned with model structures on preorders: they construct preorders $\mathcal{P},\mathcal{P}^{\prime}$ and full subcategories $W,W^{\prime}$ of $\mathcal{P}$ and $\mathcal{P}^{\prime}$ respectively, such that the pair $(\mathcal{P},W)$ satisfies the same first-order sentences as the pair $(\mathcal{P}^{\prime},W^{\prime})$, such that $\mathcal{P}^{\prime}$ admits a model structure with $W^{\prime}$ its class of weak equivalences, and such that $\mathcal{P}$ does not admit a model structure with $W$ its class of weak equivalences. Hence Droz and Zakharevich have proven, in particular, that there is no first-order characterization of when a putative class of weak equivalences {\em on a preorder} extends to a model structure.

By contrast, Corollary \ref{lettered cor b} is a first-order characterization of when a putative class of cofibrant objects and a putative class of of fibrant objects extend to a model structure on a preorder. 

For our next result, Theorem \ref{lettered thm c}, let us say that a model structure is {\em strong} if the following two conditions are satisfied:
\begin{itemize}
\item if $g\circ f$ is an acyclic cofibration, then $f$ is an acyclic cofibration, 
\item and, if $g\circ f$ is an acyclic fibration, then $g$ is an acyclic fibration.
\end{itemize}
\begin{letteredtheorem}\label{lettered thm c} (Theorem \ref{CPs induce model structures})
Suppose $\mathcal{A}$ is a finitely bicomplete preorder. Let $F$ be a monad on $\mathcal{A}$ and let $C$ be a comonad on $\mathcal{A}$. Then there exists a {\em strong} model structure on $\mathcal{A}$ whose fibrant replacement monad is $F$, and whose cofibrant replacement comonad is $C$, if and only if the following conditions are all satisfied:
\begin{itemize}
\item The two conditions from Theorem \ref{lettered thm a}.
\item If $a\leq b$ and $Cb\leq Ca$ and $Fb\leq Fa$, then $b\leq a$. 
\end{itemize}
Furthermore, we have a bijection between strong model structures on $\mathcal{A}$ and isomorphism classes of monad-comonad pairs $(F,C)$ satisfying the above conditions.
\end{letteredtheorem}

Using the fact that isomorphism classes of idempotent monads (respectively, comonads) are equivalent to reflective (respectively, coreflective) subcategories, it is straightforward to rephrase Theorem \ref{lettered thm c} as follows:
\begin{letteredcorollary}\label{lettered cor d} 
Suppose $\mathcal{A}$ is a finitely bicomplete preorder. Let $\mathcal{F},\mathcal{C}$ be subsets of $\mathcal{A}$. Then there exists a strong model structure on $\mathcal{A}$ whose fibrant objects are precisely the members of $\mathcal{F}$, and whose cofibrant objects are precisely the members of $\mathcal{C}$, if and only if all of the following conditions are satisfied:
\begin{itemize}
\item The conditions from Corollary \ref{lettered cor b}.
\item If $a\leq b$, and some lower bound for $b$ in $\mathcal{C}$ is also a lower bound for $a$, and some upper bound for $a$ in $\mathcal{F}$ is also an upper bound for $b$, then $b\leq a$.
\end{itemize}
Furthermore, we have a bijection between strong model structures on $\mathcal{A}$ and pairs of subsets $(\mathcal{F},\mathcal{C})$ of $\mathcal{A}$ satisfying the above conditions.
\end{letteredcorollary}

\subsection{Model structures constructed from topologies and from matroids}
\label{Model structures constructed from...}

Among preorders, an especially famous class is the Boolean algebras. The rest of our results concern model structures on Boolean algebras. The central examples of Boolean algebras are power sets: it is classical \cite{MR1501865} that every complete atomic Boolean algebra is the power set of a set, and in particular, every finite Boolean algebra is the power set of a set. If $S$ is a set, then idempotent monads on its power set $\mathcal{P}(S)$ are in bijection with Moore collections in $S$, i.e., collections of subsets of $S$ which are closed under arbitrary intersections. Among Moore collections on $S$, there are two types of particular interest: the topologies on $S$, and the matroids on $S$. 

Given a suitably compatible pair of topologies $(\mathcal{T}_1,\mathcal{T}_2)$ on a set $S$, Theorem \ref{lettered thm c} yields a unique strong model structure on the Boolean algebra $\mathcal{P}(S)$ by regarding $\mathcal{T}_1$ (presented via its closed sets) as a Moore collection in $S$, regarding $\mathcal{T}_2$ as a co-Moore collection in $S$, and translating these Moore and co-Moore collections into an idempotent monad and idempotent comonad on $\mathcal{P}(S)$. Unwinding the conditions of Theorem \ref{lettered thm c}, we get Proposition \ref{closure-interior compatibility}, whose statement is as follows. We will write $\Cl(U)$ for the closure of a subset $U\subseteq A$ in the topology $\mathcal{T}_1$, and we write $\Int(U)$ for the interior of $U$ in the topology $\mathcal{T}_2$. Then Proposition \ref{closure-interior compatibility} asserts that a pair $(\mathcal{T}_1,\mathcal{T}_2)$ defines a strong model structure on $\mathcal{P}(S)$ if and only if the following conditions are satisfied:
\begin{itemize}
\item $\Int(\Cl(U)) = \Cl(\Int(U))$ for all $U\subseteq S$,
\item If $U,V$ are subsets of $X$ such that $U\subseteq \Int(V)$ and $\Cl(U)\subseteq V$, then $\Cl(U)\subseteq \Int(V)$.
\item
If $U,V$ are subsets of $X$ such that $U\subseteq V$ and $\Cl(U) = \Cl(V)$ and $\Int(U) = \Int(V)$, then $U=V$.
\end{itemize}
In that model structure on $\mathcal{P}(S)$, $\Cl$ is fibrant replacement, and $\Int$ is cofibrant replacement. The homotopy category of the model structure is the sub-poset of $\mathcal{P}(S)$ consisting of those subsets of $S$ which are both closed in topology $\mathcal{T}_1$ and open in topology $\mathcal{T}_2$. We call such a model structure, arising from a pair of topologies, {\em topological}.

Similarly, given a pair of matroids\footnote{See Definition \ref{def of matroid} for the definition of a matroid, if desired.} $(\mathcal{M}_1,\mathcal{M}_2)$ with the same underlying set $S$, Proposition \ref{model structs from matroids} tells us precisely when we get a corresponding strong model structure on the power set $\mathcal{P}(S)$. Write $\Cl$ for the closure operator in $\mathcal{M}_1$, i.e., $\Cl(U)$ is the intersection of all the flats in $\mathcal{M}_1$ containing $U$. Similarly, write $\Int(U)$ for the
union of the complements of the flats in $\mathcal{M}_2$ containing the complement of $U$. Then we get a corresponding strong model structure on $\mathcal{P}(S)$ if and only if the same two conditions as in the topological case, above, are satisfied. The homotopy category of the model structure is the sub-poset of $\mathcal{P}(S)$ consisting of those subsets $U$ of $S$ such that $U$ is a flat in the matroid $\mathcal{M}_1$, and $U$ is also a cycle in the dual matroid $\mathcal{M}_2^*$ of $\mathcal{M}_2$. We call these model structures, arising from pairs of matroids, {\em matroidal}. Similarly, those that arise from pairs of simple matroids, i.e., ``geometries'' in the sense of \cite{MR1783451}, will be called {\em geometric}.

\subsection{Case studies of all model structures on small Boolean algebras}
\label{case studies...}

With the above structural results and constructions in hand, we are in a position to carry out explicit calculations (with computer assistance) of all the model structures on some small Boolean algebras, and to really understand what we are looking at.
In \cref{Model structures arising from topologies} and in \cref{Explicit examples}, we give the results of such calculations.
Let $\mathcal{P}_n$ denote the power set of an $n$-element set.
We show that there are:
\begin{itemize}
\item 3 model structures on $\mathcal{P}_1$. All are matroidal, hence also strong, but none are topological or geometric except the discrete model structure (i.e., the unique model structure in which the only weak equivalences are isomorphisms), which is both topological and geometric.
\item 23 model structures on $\mathcal{P}_2$, of which 17 are strong, 9 are topological, 11 are matroidal, and one---the discrete model structure---is geometric.
\item 1026 model structures on $\mathcal{P}_3$, of which 377 are strong, 84 are topological, 50 are matroidal, and 4 are geometric.
\end{itemize}
Our methods yield much more detailed information than merely these enumerations. The following diagram depicts all 23 model structures on $\mathcal{P}_2$:
\begin{equation}
\label{n2 model structs}
\tikzset{
 d00/.pic={
\draw[lightgray, very thin] (0,-1) -- (1,0) -- (0,1) -- (-1,0) -- cycle; 
\draw[lightgray, very thin] (0,-1) -- (0,1);
\draw[green] (0,-1) node[circle,minimum size=0cm,fill=green!40] {};
\path (-.5,0.5) node[color=black] {$fw$} ;
\path (.5,0.5) node[color=black] {$fw$} ;
\path (0,0) node[color=black] {$fw$} ;
\path (-.5,-0.5) node[color=black] {$fw$} ;
\path (.5,-0.5) node[color=black] {$fw$};
\path (0,2) node [color=black] {$0,0_*$};
}}
\tikzset{
 d04/.pic={
\draw[lightgray, very thin] (0,-1) -- (1,0) -- (0,1) -- (-1,0) -- cycle; 
\draw[lightgray, very thin] (0,-1) -- (0,1);
\draw[green] (0,-1) node[circle,minimum size=0cm,fill=green!40] {};
\draw[green] (0,1) node[circle,minimum size=0cm,fill=green!40] {};
\path (-.5,0.5) node[color=black] {$cf$} ;
\path (.5,0.5) node[color=black] {$cf$} ;
\path (0,0) node[color=black] {$cf$} ;
\path (-.5,-0.5) node[color=black] {$fw$} ;
\path (.5,-0.5) node[color=black] {$fw$};
\path (0,2) node [color=black] {$0,4_*$};
}}
\tikzset{
 d05/.pic={
\draw[lightgray, very thin] (0,-1) -- (1,0) -- (0,1) -- (-1,0) -- cycle; 
\draw[lightgray, very thin] (0,-1) -- (0,1);
\draw[green] (0,-1) node[circle,minimum size=0cm,fill=green!40] {};
\draw[green] (-1,0) node[circle,minimum size=0cm,fill=green!40] {};
\path (-.5,0.5) node[color=black] {$fw$} ;
\path (.5,0.5) node[color=black] {$cf$} ;
\path (0,0) node[color=black] {$f$} ;
\path (-.5,-0.5) node[color=black] {$cf$} ;
\path (.5,-0.5) node[color=black] {$fw$};
\path (0,2) node [color=black] {$0,5_*$};
}}
\tikzset{
 d06/.pic={
\draw[lightgray, very thin] (0,-1) -- (1,0) -- (0,1) -- (-1,0) -- cycle; 
\draw[lightgray, very thin] (0,-1) -- (0,1);
\draw[green] (0,-1) node[circle,minimum size=0cm,fill=green!40] {};
\draw[green] (-1,0) node[circle,minimum size=0cm,fill=green!40] {};
\draw[green] (0,1) node[circle,minimum size=0cm,fill=green!40] {};
\path (-.5,0.5) node[color=black] {$cf$} ;
\path (.5,0.5) node[color=black] {$cf$} ;
\path (0,0) node[color=black] {$cf$} ;
\path (-.5,-0.5) node[color=black] {$cf$} ;
\path (.5,-0.5) node[color=black] {$fw$};
\path (0,2) node [color=black] {$0,6_*$};
}}
\tikzset{
 d07/.pic={
\draw[lightgray, very thin] (0,-1) -- (1,0) -- (0,1) -- (-1,0) -- cycle; 
\draw[lightgray, very thin] (0,-1) -- (0,1);
\draw[green] (0,-1) node[circle,minimum size=0cm,fill=green!40] {};
\draw[green] (1,0) node[circle,minimum size=0cm,fill=green!40] {};
\path (-.5,0.5) node[color=black] {$cf$} ;
\path (.5,0.5) node[color=black] {$fw$} ;
\path (0,0) node[color=black] {$f$} ;
\path (-.5,-0.5) node[color=black] {$fw$} ;
\path (.5,-0.5) node[color=black] {$cf$};
\path (0,2) node [color=black] {$0,7_*$};
}}
\tikzset{
 d08/.pic={
\draw[lightgray, very thin] (0,-1) -- (1,0) -- (0,1) -- (-1,0) -- cycle; 
\draw[lightgray, very thin] (0,-1) -- (0,1);
\draw[green] (0,-1) node[circle,minimum size=0cm,fill=green!40] {};
\draw[green] (1,0) node[circle,minimum size=0cm,fill=green!40] {};
\draw[green] (0,1) node[circle,minimum size=0cm,fill=green!40] {};
\path (-.5,0.5) node[color=black] {$cf$} ;
\path (.5,0.5) node[color=black] {$cf$} ;
\path (0,0) node[color=black] {$cf$} ;
\path (-.5,-0.5) node[color=black] {$fw$} ;
\path (.5,-0.5) node[color=black] {$cf$};
\path (0,2) node [color=black] {$0,8_*$};
}}
\tikzset{
 d09/.pic={
\draw[lightgray, very thin] (0,-1) -- (1,0) -- (0,1) -- (-1,0) -- cycle; 
\draw[lightgray, very thin] (0,-1) -- (0,1);
\draw[green] (0,-1) node[circle,minimum size=0cm,fill=green!40] {};
\draw[green] (-1,0) node[circle,minimum size=0cm,fill=green!40] {};
\draw[green] (1,0) node[circle,minimum size=0cm,fill=green!40] {};
\draw[green] (0,1) node[circle,minimum size=0cm,fill=green!40] {};
\path (-.5,0.5) node[color=black] {$cf$} ;
\path (.5,0.5) node[color=black] {$cf$} ;
\path (0,0) node[color=black] {$cf$} ;
\path (-.5,-0.5) node[color=black] {$cf$} ;
\path (.5,-0.5) node[color=black] {$cf$};
\path (0,2) node [color=black] {$0,9_*$};
}}
\tikzset{
 d11/.pic={
\draw[lightgray, very thin] (0,-1) -- (1,0) -- (0,1) -- (-1,0) -- cycle; 
\draw[lightgray, very thin] (0,-1) -- (0,1);
\draw[green] (0,-1) node[circle,minimum size=0cm,fill=green!40] {};
\path (-.5,0.5) node[color=black] {$cw$} ;
\path (.5,0.5) node[color=black] {$fw$} ;
\path (0,0) node[color=black] {$fw$} ;
\path (-.5,-0.5) node[color=black] {$fw$} ;
\path (.5,-0.5) node[color=black] {$fw$};
\path (0,2) node [color=black] {$1,1$};
}}
\tikzset{
 d16/.pic={
\draw[lightgray, very thin] (0,-1) -- (1,0) -- (0,1) -- (-1,0) -- cycle; 
\draw[lightgray, very thin] (0,-1) -- (0,1);
\draw[green] (0,-1) node[circle,minimum size=0cm,fill=green!40] {};
\draw[green] (0,1) node[circle,minimum size=0cm,fill=green!40] {};
\path (-.5,0.5) node[color=black] {$cw$} ;
\path (.5,0.5) node[color=black] {$cf$} ;
\path (0,0) node[color=black] {$cf$} ;
\path (-.5,-0.5) node[color=black] {$cf$} ;
\path (.5,-0.5) node[color=black] {$fw$};
\path (0,2) node [color=black] {$1,6_*$};
}}
\tikzset{
 d19/.pic={
\draw[lightgray, very thin] (0,-1) -- (1,0) -- (0,1) -- (-1,0) -- cycle; 
\draw[lightgray, very thin] (0,-1) -- (0,1);
\draw[green] (0,-1) node[circle,minimum size=0cm,fill=green!40] {};
\draw[green] (1,0) node[circle,minimum size=0cm,fill=green!40] {};
\draw[green] (0,1) node[circle,minimum size=0cm,fill=green!40] {};
\path (-.5,0.5) node[color=black] {$cw$} ;
\path (.5,0.5) node[color=black] {$cf$} ;
\path (0,0) node[color=black] {$cf$} ;
\path (-.5,-0.5) node[color=black] {$cf$} ;
\path (.5,-0.5) node[color=black] {$cf$};
\path (0,2) node [color=black] {$1,9_*$};
}}
\tikzset{
 d22/.pic={
\draw[lightgray, very thin] (0,-1) -- (1,0) -- (0,1) -- (-1,0) -- cycle; 
\draw[lightgray, very thin] (0,-1) -- (0,1);
\draw[green] (0,-1) node[circle,minimum size=0cm,fill=green!40] {};
\path (-.5,0.5) node[color=black] {$fw$} ;
\path (.5,0.5) node[color=black] {$cw$} ;
\path (0,0) node[color=black] {$fw$} ;
\path (-.5,-0.5) node[color=black] {$fw$} ;
\path (.5,-0.5) node[color=black] {$fw$};
\path (0,2) node [color=black] {$2,2$};
}}
\tikzset{
 d28/.pic={
\draw[lightgray, very thin] (0,-1) -- (1,0) -- (0,1) -- (-1,0) -- cycle; 
\draw[lightgray, very thin] (0,-1) -- (0,1);
\draw[green] (0,-1) node[circle,minimum size=0cm,fill=green!40] {};
\draw[green] (0,1) node[circle,minimum size=0cm,fill=green!40] {};
\path (-.5,0.5) node[color=black] {$cf$} ;
\path (.5,0.5) node[color=black] {$cw$} ;
\path (0,0) node[color=black] {$cf$} ;
\path (-.5,-0.5) node[color=black] {$fw$} ;
\path (.5,-0.5) node[color=black] {$cf$};
\path (0,2) node [color=black] {$2,8_*$};
}}
\tikzset{
 d29/.pic={
\draw[lightgray, very thin] (0,-1) -- (1,0) -- (0,1) -- (-1,0) -- cycle; 
\draw[lightgray, very thin] (0,-1) -- (0,1);
\draw[green] (0,-1) node[circle,minimum size=0cm,fill=green!40] {};
\draw[green] (-1,0) node[circle,minimum size=0cm,fill=green!40] {};
\draw[green] (0,1) node[circle,minimum size=0cm,fill=green!40] {};
\path (-.5,0.5) node[color=black] {$cf$} ;
\path (.5,0.5) node[color=black] {$cw$} ;
\path (0,0) node[color=black] {$cf$} ;
\path (-.5,-0.5) node[color=black] {$cf$} ;
\path (.5,-0.5) node[color=black] {$cf$};
\path (0,2) node [color=black] {$2,9_*$};
}}
\tikzset{
 d33/.pic={
\draw[lightgray, very thin] (0,-1) -- (1,0) -- (0,1) -- (-1,0) -- cycle; 
\draw[lightgray, very thin] (0,-1) -- (0,1);
\draw[green] (0,-1) node[circle,minimum size=0cm,fill=green!40] {};
\path (-.5,0.5) node[color=black] {$cw$} ;
\path (.5,0.5) node[color=black] {$cw$} ;
\path (0,0) node[color=black] {$fw$} ;
\path (-.5,-0.5) node[color=black] {$fw$} ;
\path (.5,-0.5) node[color=black] {$fw$};
\path (0,2) node [color=black] {$3,3$};
}}
\tikzset{
 d39/.pic={
\draw[lightgray, very thin] (0,-1) -- (1,0) -- (0,1) -- (-1,0) -- cycle; 
\draw[lightgray, very thin] (0,-1) -- (0,1);
\draw[green] (0,-1) node[circle,minimum size=0cm,fill=green!40] {};
\draw[green] (0,1) node[circle,minimum size=0cm,fill=green!40] {};
\path (-.5,0.5) node[color=black] {$cw$} ;
\path (.5,0.5) node[color=black] {$cw$} ;
\path (0,0) node[color=black] {$cf$} ;
\path (-.5,-0.5) node[color=black] {$cf$} ;
\path (.5,-0.5) node[color=black] {$cf$};
\path (0,2) node [color=black] {$3,9_*$};
}}
\tikzset{
 d44/.pic={
\draw[lightgray, very thin] (0,-1) -- (1,0) -- (0,1) -- (-1,0) -- cycle; 
\draw[lightgray, very thin] (0,-1) -- (0,1);
\draw[green] (0,1) node[circle,minimum size=0cm,fill=green!40] {};
\path (-.5,0.5) node[color=black] {$cw$} ;
\path (.5,0.5) node[color=black] {$cw$} ;
\path (0,0) node[color=black] {$cw$} ;
\path (-.5,-0.5) node[color=black] {$fw$} ;
\path (.5,-0.5) node[color=black] {$fw$};
\path (0,2) node [color=black] {$4,4$};
}}
\tikzset{
 d55/.pic={
\draw[lightgray, very thin] (0,-1) -- (1,0) -- (0,1) -- (-1,0) -- cycle; 
\draw[lightgray, very thin] (0,-1) -- (0,1);
\draw[green] (-1,0) node[circle,minimum size=0cm,fill=green!40] {};
\path (-.5,0.5) node[color=black] {$fw$} ;
\path (.5,0.5) node[color=black] {$cw$} ;
\path (0,0) node[color=black] {$w$} ;
\path (-.5,-0.5) node[color=black] {$cw$} ;
\path (.5,-0.5) node[color=black] {$fw$};
\path (0,2) node [color=black] {$5,5_*$};
}}
\tikzset{
 d59/.pic={
\draw[lightgray, very thin] (0,-1) -- (1,0) -- (0,1) -- (-1,0) -- cycle; 
\draw[lightgray, very thin] (0,-1) -- (0,1);
\draw[green] (-1,0) node[circle,minimum size=0cm,fill=green!40] {};
\draw[green] (0,1) node[circle,minimum size=0cm,fill=green!40] {};
\path (-.5,0.5) node[color=black] {$cf$} ;
\path (.5,0.5) node[color=black] {$cw$} ;
\path (0,0) node[color=black] {$c$} ;
\path (-.5,-0.5) node[color=black] {$cw$} ;
\path (.5,-0.5) node[color=black] {$cf$};
\path (0,2) node [color=black] {$5,9_*$};
}}
\tikzset{
 d66/.pic={
\draw[lightgray, very thin] (0,-1) -- (1,0) -- (0,1) -- (-1,0) -- cycle; 
\draw[lightgray, very thin] (0,-1) -- (0,1);
\draw[green] (0,1) node[circle,minimum size=0cm,fill=green!40] {};
\path (-.5,0.5) node[color=black] {$cw$} ;
\path (.5,0.5) node[color=black] {$cw$} ;
\path (0,0) node[color=black] {$cw$} ;
\path (-.5,-0.5) node[color=black] {$cw$} ;
\path (.5,-0.5) node[color=black] {$fw$};
\path (0,2) node [color=black] {$6,6$};
}}
\tikzset{
 d77/.pic={
\draw[lightgray, very thin] (0,-1) -- (1,0) -- (0,1) -- (-1,0) -- cycle; 
\draw[lightgray, very thin] (0,-1) -- (0,1);
\draw[green] (1,0) node[circle,minimum size=0cm,fill=green!40] {};
\path (-.5,0.5) node[color=black] {$cw$} ;
\path (.5,0.5) node[color=black] {$fw$} ;
\path (0,0) node[color=black] {$w$} ;
\path (-.5,-0.5) node[color=black] {$fw$} ;
\path (.5,-0.5) node[color=black] {$cw$};
\path (0,2) node [color=black] {$7,7_*$};
}}
\tikzset{
 d79/.pic={
\draw[lightgray, very thin] (0,-1) -- (1,0) -- (0,1) -- (-1,0) -- cycle; 
\draw[lightgray, very thin] (0,-1) -- (0,1);
\draw[green] (1,0) node[circle,minimum size=0cm,fill=green!40] {};
\draw[green] (0,1) node[circle,minimum size=0cm,fill=green!40] {};
\path (-.5,0.5) node[color=black] {$cw$} ;
\path (.5,0.5) node[color=black] {$cf$} ;
\path (0,0) node[color=black] {$c$} ;
\path (-.5,-0.5) node[color=black] {$cf$} ;
\path (.5,-0.5) node[color=black] {$cw$};
\path (0,2) node [color=black] {$7,9_*$};
}}
\tikzset{
 d88/.pic={
\draw[lightgray, very thin] (0,-1) -- (1,0) -- (0,1) -- (-1,0) -- cycle; 
\draw[lightgray, very thin] (0,-1) -- (0,1);
\draw[green] (0,1) node[circle,minimum size=0cm,fill=green!40] {};
\path (-.5,0.5) node[color=black] {$cw$} ;
\path (.5,0.5) node[color=black] {$cw$} ;
\path (0,0) node[color=black] {$cw$} ;
\path (-.5,-0.5) node[color=black] {$fw$} ;
\path (.5,-0.5) node[color=black] {$cw$};
\path (0,2) node [color=black] {$8,8$};
}}
\tikzset{
 d99/.pic={
\draw[lightgray, very thin] (0,-1) -- (1,0) -- (0,1) -- (-1,0) -- cycle; 
\draw[lightgray, very thin] (0,-1) -- (0,1);
\draw[green] (0,1) node[circle,minimum size=0cm,fill=green!40] {};
\path (-.5,0.5) node[color=black] {$cw$} ;
\path (.5,0.5) node[color=black] {$cw$} ;
\path (0,0) node[color=black] {$cw$} ;
\path (-.5,-0.5) node[color=black] {$cw$} ;
\path (.5,-0.5) node[color=black] {$cw$};
\path (0,2) node [color=black] {$9,9_*$};
}}
\scalebox{1.7}{
\begin{tikzpicture} [shorten <= 0.4cm,shorten >= 0.4cm,
 d00/.style={path picture={\pic at ([yshift=-0.1cm]path picture bounding box) {d00};}},
 d04/.style={path picture={\pic at ([yshift=-0.1cm]path picture bounding box) {d04};}},
 d05/.style={path picture={\pic at ([yshift=-0.1cm]path picture bounding box) {d05};}},
 d06/.style={path picture={\pic at ([yshift=-0.1cm]path picture bounding box) {d06};}},
 d07/.style={path picture={\pic at ([yshift=-0.1cm]path picture bounding box) {d07};}},
 d08/.style={path picture={\pic at ([yshift=-0.1cm]path picture bounding box) {d08};}},
 d09/.style={path picture={\pic at ([yshift=-0.1cm]path picture bounding box) {d09};}},
 d11/.style={path picture={\pic at ([yshift=-0.1cm]path picture bounding box) {d11};}},
 d16/.style={path picture={\pic at ([yshift=-0.1cm]path picture bounding box) {d16};}},
 d19/.style={path picture={\pic at ([yshift=-0.1cm]path picture bounding box) {d19};}},
 d22/.style={path picture={\pic at ([yshift=-0.1cm]path picture bounding box) {d22};}},
 d28/.style={path picture={\pic at ([yshift=-0.1cm]path picture bounding box) {d28};}},
 d29/.style={path picture={\pic at ([yshift=-0.1cm]path picture bounding box) {d29};}},
 d33/.style={path picture={\pic at ([yshift=-0.1cm]path picture bounding box) {d33};}},
 d39/.style={path picture={\pic at ([yshift=-0.1cm]path picture bounding box) {d39};}},
 d44/.style={path picture={\pic at ([yshift=-0.1cm]path picture bounding box) {d44};}},
 d55/.style={path picture={\pic at ([yshift=-0.1cm]path picture bounding box) {d55};}},
 d59/.style={path picture={\pic at ([yshift=-0.1cm]path picture bounding box) {d59};}},
 d66/.style={path picture={\pic at ([yshift=-0.1cm]path picture bounding box) {d66};}},
 d77/.style={path picture={\pic at ([yshift=-0.1cm]path picture bounding box) {d77};}},
 d79/.style={path picture={\pic at ([yshift=-0.1cm]path picture bounding box) {d79};}},
 d88/.style={path picture={\pic at ([yshift=-0.1cm]path picture bounding box) {d88};}},
 d99/.style={path picture={\pic at ([yshift=-0.1cm]path picture bounding box) {d99};}},
every pic/.style={scale=.2,every node/.style={scale=0.4}},
] 
\draw ( 0,-4) node[d00,minimum size=0.8cm,circle,draw,color=orange,fill=orange!10] {};
\draw ( 0,-2) node[d04,minimum size=0.8cm,circle,draw,color=orange,shading=true, left color=blue!10, right color=orange!10] {};
\draw ( 2,-2) node[d05,minimum size=0.8cm,circle,draw,color=orange,fill=orange!10] {};
\draw ( 1,-1) node[d06,minimum size=0.8cm,circle,draw,color=orange,fill=blue!10] {};
\draw (-2,-2) node[d07,minimum size=0.8cm,circle,draw,color=orange,fill=orange!10] {};
\draw (-1,-1) node[d08,minimum size=0.8cm,circle,draw,color=orange,fill=blue!10] {};
\draw ( 0, 0) node[d09,minimum size=0.8cm,circle,draw,color=red,fill=red!15] {};
\draw ( 3,-1) node[d11,minimum size=0.8cm,circle,draw,color=yellow] {};
\draw ( 2, 0) node[d16,minimum size=0.8cm,circle,draw,color=yellow,fill=blue!10] {};
\draw ( 1, 1) node[d19,minimum size=0.8cm,circle,draw,color=orange,fill=blue!10] {};
\draw (-3,-1) node[d22,minimum size=0.8cm,circle,draw,color=yellow] {};
\draw (-2, 0) node[d28,minimum size=0.8cm,circle,draw,color=yellow,fill=blue!10] {};
\draw (-1, 1) node[d29,minimum size=0.8cm,circle,draw,color=orange,fill=blue!10] {};
\draw ( 1, 2) node[d33,minimum size=0.8cm,circle,draw,color=yellow] {};
\draw ( 0, 2) node[d39,minimum size=0.8cm,circle,draw,color=orange,shading=true, left color=blue!10, right color=orange!10] {};
\draw (-1,-2) node[d44,minimum size=0.8cm,circle,draw,color=yellow] {};
\draw ( 4, 2) node[d55,minimum size=0.8cm,circle,draw,color=yellow,fill=orange!10] {};
\draw (-2, 2) node[d59,minimum size=0.8cm,circle,draw,color=orange,fill=orange!10] {};
\draw ( 3, 1) node[d66,minimum size=0.8cm,circle,draw,color=yellow] {};
\draw ( 4,-2) node[d77,minimum size=0.8cm,circle,draw,color=yellow,fill=orange!10] {};
\draw ( 2, 2) node[d79,minimum size=0.8cm,circle,draw,color=orange,fill=orange!10] {};
\draw (-3, 1) node[d88,minimum size=0.8cm,circle,draw,color=yellow] {};
\draw ( 0, 4) node[d99,minimum size=0.8cm,circle,draw,color=orange,fill=orange!10] {};
\draw [->,color=black] (-2, 2) -- ( 0, 4);
\draw [->,color=black] ( 2, 2) -- ( 0, 4);
\draw [->,color=black] ( 0, 2) -- ( 0, 4);
\draw [->,color=black] (-1, 1) -- (-2, 2);
\draw [->,color=black] (-1, 1) -- ( 0, 2);
\draw [->,color=black] ( 1, 1) -- ( 2, 2);
\draw [->,color=black] ( 1, 1) -- ( 0, 2);
\draw [->,color=black] (-2, 0) -- (-3, 1);
\draw [->,color=black] ( 0, 0) -- (-1, 1);
\draw [->,color=black] ( 0, 0) -- ( 1, 1);
\draw [->,color=black] (-1,-1) -- (-2, 0);
\draw [->,color=black] ( 1,-1) -- ( 2, 0);
\draw [->,color=black] ( 2, 0) -- ( 3, 1);
\draw [->,color=black] ( 0,-2) .. controls (-0.5,-2.5) and (-0.5,-2.5) .. (-1,-2);
\draw [->,color=black] ( 2,-2) .. controls (3,-3) and (5,1) .. ( 4, 2) ;
\draw [->,color=black] (-2,-2) .. controls (-1,-3) and (3,-3) .. ( 4,-2);
\draw [->,dashed,color=black] (-2, 2) .. controls (-1, 3) and (3, 3) .. ( 4, 2);
\draw [->,dashed,color=black] ( 2, 2) .. controls ( 3, 3) and (5,-1) .. ( 4,-2);
\draw [->,dashed,color=black] ( 0, 2) .. controls (0.5, 2.5) and (0.5, 2.5) .. ( 1, 2);
\draw [->,dashed,color=black] (-1, 1) -- (-2, 0);
\draw [->,dashed,color=black] (-2, 0) -- (-3,-1);
\draw [->,dashed,color=black] ( 0, 0) -- (-1,-1);
\draw [->,dashed,color=black] ( 0, 0) -- ( 1,-1);
\draw [->,dashed,color=black] ( 2, 0) -- ( 3,-1);
\draw [->,dashed,color=black] (-1,-1) -- (-2,-2);
\draw [->,dashed,color=black] ( 1,-1) -- ( 2,-2);
\draw [->,dashed,color=black] (-1,-1) -- ( 0,-2);
\draw [->,dashed,color=black] ( 2,-2) -- ( 0,-4);
\draw [->,dashed,color=black] ( 1, 1) -- ( 2, 0);
\draw [->,dashed,color=black] ( 1,-1) -- ( 0,-2);
\draw [->,dashed,color=black] (-2,-2) -- ( 0,-4);
\draw [->,dashed,color=black] ( 0,-2) -- ( 0,-4);
\end{tikzpicture}
}
\end{equation}

Notations in figure \eqref{n2 model structs} are as follows.
\begin{itemize}
\item In figure \eqref{n2 model structs}, each circle represents a single model structure. 
\item The ordered pair $m,n$ of integers at the top of each circle indicates which pair of factorization systems gives that model structure. This uses our numbering scheme for the (precisely 10) factorization systems on $\mathcal{P}_2$, given later in the paper, in \eqref{n2 fact systems}. For example, $0,9$ indicates the model structure whose acyclic-cofibration-and-then-fibration factorization system is number $0$ from figure \eqref{n2 fact systems}, and whose cofibration-and-then-acyclic-fibration factorization system is number $9$ from figure \eqref{n2 fact systems}. 
\item An asterisk is written as a subscript below the ordered pair $m,n$ if and only if the corresponding model structure is strong. 
\item Background colors are used as follows:
\begin{itemize}
\item a white background indicates a non-strong model structure,
\item a blue background indicates a model structure which is topological (hence also strong), 
\item an orange background indicates a model structure which is matroidal (hence also strong), 
\item a blue and orange background indicates a model structure which is both topological and matroidal,
\item and a red background indicates a model structure which is both geometric and topological.
\end{itemize}
\item The Boolean algebra $\mathcal{P}_2$ is the power set of a set $\{ a,b\}$. We may graphically depict the four elements and five non-identity morphisms in $\mathcal{P}_2$ with the diagram
\begin{equation}\label{diag 02300a}\xymatrix{
 & \{a,b\} & \\
\{a\} \ar[ur] & & \{b\} \ar[ul] \\
 & \emptyset . \ar[ur]\ar[uu]\ar[ul] & 
}\end{equation}
In each circle in \eqref{n2 model structs}, a miniature version of diagram \eqref{diag 02300a} is marked to indicate which morphisms are cofibrations ($c$), weak equivalences ($w$), or fibrations ($f$); for example, \tikzset{
 d79/.pic={
\draw[lightgray, very thin] (0,-1) -- (1,0) -- (0,1) -- (-1,0) -- cycle; 
\draw[lightgray, very thin] (0,-1) -- (0,1);
\path (-.5,0.5) node[color=black] {$cw$} ;
\path (.5,0.5) node[color=black] {$cf$} ;
\path (0,0) node[color=black] {$c$} ;
\path (-.5,-0.5) node[color=black] {$cf$} ;
\path (.5,-0.5) node[color=black] {$cw$};
}}
\begin{tikzpicture} [shorten <= 0.4cm,shorten >= 0.4cm,
 d79/.style={path picture={\pic at ([yshift=-0.1cm]path picture bounding box) {d79};}},
every pic/.style={scale=.2,every node/.style={scale=0.4}},
] 
\draw ( 0, 0) node[d79,minimum size=0.8cm] {};
\end{tikzpicture}
indicates the model structure in which all maps are cofibrations, $\emptyset\rightarrow \{ a\}$ and $\{b\}\rightarrow \{a,b\}$ are each also fibrations, and $\emptyset\rightarrow \{ b\}$ and $\{a\}\rightarrow \{a,b\}$ are each also weak equivalences. 
\item In each circle, the green dots mark those objects of $P(\{a,b\})$ which are bifibrant in the corresponding model structure. Since the homotopy category of a model structure on a preorder $\mathcal{A}$ is simply the full subcategory of $\mathcal{A}$ containing the bifibrant objects, the green dots tell us, for each model structure on $P(\{a,b\})$, what the homotopy category of that model structure is: namely, it is the full subcategory (i.e., the sub-poset) containing the objects of $P(\{a,b\})$ marked by the green dots.
\item A solid arrow from one model structure to another indicates that the latter is a Bousfield localization\footnote{Recall that, given model structures $\mathcal{M}_1,\mathcal{M}_2$ on a category $\mathcal{A}$, we say that $\mathcal{M}_2$ is a {\em Bousfield localization} (respectively, {\em Bousfield colocalization}) of $\mathcal{M}_1$ if $\mathcal{M}_1$ and $\mathcal{M}_2$ share the same class of cofibrations (respectively, fibrations), and the class of weak equivalences in $\mathcal{M}_1$ is contained in the class of weak equivalences of $\mathcal{M}_2$.} of the former; a dashed arrow indicates that the latter is a Bousfield colocalization (i.e., a ``right Bousfield localization'') of the former. We have only drawn a {\em minimal} such set of arrows; for example, we have a solid arrow from $(0,9)$ to $(1,9)$, and a solid arrow from $(1,9)$ to $(7,9)$, so it is implied that model structure $(7,9)$ is also a Bousfield localization of model structure $(0,9)$ without our needing to draw an additional arrow between them.
\item The model structure circled in red (i.e., $(0,9)$) is the discrete model structure, i.e., the model structure in which all maps are cofibrations and fibrations, and only identity maps are weak equivalences.
\item The color of the circles represents distance from the discrete model structure, in the following sense: the discrete model structure is circled in red. The model structures which can be obtained by a single localization {\em or} a single colocalization from the discrete model structure are colored in orange. The model structures which can be obtained by a length $2$ alternating sequence of Bousfield localizations and colocalizations are colored in yellow. There are no model structures on $P(\{a,b\})$ requiring alternating sequences of localizations and colocalizations of length $\geq 3$ for their construction, so no further colors are required.
\end{itemize}
Given a category $\mathcal{A}$, one could try to make a diagram in which we have one node for each model structure on $\mathcal{A}$, and in which we have arrows between nodes of one color to indicate which nodes are Bousfield localizations of which others, and arrows of some other color to indicate which nodes are Bousfield colocalizations of which others. Around 2014, the first author began calling such diagrams ``Bousfield quivers'' after calculating a few of them. The term first appeared in print in Corrigan-Salter's paper \cite{MR3513962}. Diagram \eqref{n2 model structs} is a Bousfield quiver, augmented with various decorations that are meaningful in the case of model structures on Boolean algebras. The first author posted \cite{moanswer} the Bousfield quiver \eqref{n2 model structs} in May 2025. An independent partial calcuation of the same Bousfield quiver was described in the November 2025 preprint \cite{leftandright}---see Figure 5 in \cite{leftandright}, where some (but not all) of the Bousfield localization and colocalizations are depicted.

Some structural features of the collection of model structures on a category can be read off from its Bousfield quiver. Some of these structural features are not at all obvious. For example, from \eqref{n2 model structs} we have the curious observation that every strong model structure on $\mathcal{P}_2$ is either topological or matroidal (or both), and that every geometric model structure on $\{a,b\}$ is also topological (and discrete). We see also that every model structure on $\mathcal{P}_2$ is obtainable from the discrete model structure by iterated Bousfield localization and colocalization---in fact, it suffices to take a sequence of length $2$, i.e., a single localization followed by a single colocalization, or the reverse\footnote{This has some overlap with an observation that was made in the very recent preprint \cite{leftandright}, although neither observation is directly a consequence of the other. In \cite[Theorem 4.10]{leftandright} it is shown that every model structure on the poset $[m]\times [n+1]$ {\em whose acyclic fibrations are retractile} (Definition \ref{retractile and sectile classes}) is obtainable from the discrete model structure by a finite sequence of Bousfield localizations and colocalizations.}. For short, we say that the Bousfield quiver of $\mathcal{P}_2$ is {\em connected}.

In \cref{Explicit examples}, we also provide the Bousfield quiver for $\mathcal{P}_1$ and for $\mathcal{P}_3$, although the latter can only be depicted in a highly condensed form, since it has 1026 vertices. The case of $\mathcal{P}_3$ begins to clarify which structural properties of the set of model structures on $\mathcal{P}_2$ are coincidental, and which ones perhaps generalize to model structures on other Boolean algebras (or perhaps other categories, more broadly):
\begin{itemize}
\item It is not true that every strong model structure on $\mathcal{P}_3$ is either topological or matroidal.
\item It is not true that the Bousfield quiver of $\mathcal{P}_3$ is connected. In \cref{P3}, we explain that there are 765 model structures in the connected component of the discrete model structure, i.e., ``the main component.'' The remaining 261 model structures are spread out across 219 connected components, each of which is structured in a very simple way, described in \cref{P3}. Essentially ``everything complicated'' happens in the main component. 
\item Curiously, all the model structures on $\mathcal{P}_3$ outside the main component are non-strong. Hence all the topological and matroidal model structures on $\mathcal{P}_3$ occur in the main component, i.e., each such model structure is obtainable by finitely many Bousfield localizations and colocalizations, starting from the discrete model structure.
\item Given a model structure $\mathcal{M}$ on a category $\mathcal{A}$, let us say that the {\em distance from the discrete model structure to $\mathcal{M}$} is the length of the shortest chain of Bousfield localizations and colocalizations which begins with the discrete model structure and ends at $\mathcal{M}$. Let us say that
the {\em radius of the main component in the Bousfield quiver of $\mathcal{A}$} is the supremum of the finite distances from the discrete model structure to other model structures on $\mathcal{A}$. From our calculations in \cref{Explicit examples}, we observe that, for $n\in \{1,2,3\}$, the radius of the main component in the Bousfield quiver of $\mathcal{P}_n$ is precisely $n$. We wonder if this is true for larger values of $n$. An explicit calculation in the case $n=4$ is out of reach without improved algorithms---see discussion in \cref{P of ab}.
\end{itemize}

Boolean algebras equipped with idempotent (co)monads arise in the usual algebraic semantics for modal propositional logic (see \cite[chapter 10]{MR1858927} for a nice textbook treatment). The first author wrote an appendix to this paper about the modal-logical interpretation of model structures on Boolean algebras, using the results of this paper. The audience for that appendix would presumably be the set of mathematicians who are interested in both model categories and modal logic. We are unsure whether that set of mathematicians contains any elements which are not on the list of authors of this paper, so the first author has elected to present that appendix as a separate download\footnote{\url{https://asalch.wayne.edu/modal_appendix_1.pdf}} from his webpage, rather than as part of the paper itself. 

The paper is organized as follows. In \cref{basic properties}, we lay out basic properties of model structures on preorders, and define retractile and sectile factorization systems. The main theoretical results appear in \cref{model structures from compatible pairs}, where we prove Theorems \ref{characterization thm} and \ref{CPs induce model structures}. 
In \cref{factorization bracket pairs section}, we give an equivalent description of model structures on a preorder in terms of factorization–bracket pairs. This definition is more suitable for computational implementation. As an application for these results, we describe topological and matroidal model structures on finite Boolean algebras in \cref{Strong model structures on finite Boolean algebras}. Finally, in \cref{Explicit examples}, we describe the Bousfield quiver of all model structures on a finite Boolean algebra and record several basic observations about such quivers.

\subsection{Prior related work}
\label{Previous related work}

Here are a few earlier papers on the problem of enumerating the model structures or factorization systems on certain preorders. None of these papers share the present paper's method of constructing and classifying model structures via their (co)fibrant replacement (co)monads, nor our focus on calculating the model structures on Boolean algebras, so while there is some overlap in broad subject matter, there is not much overlap in methods, results, or considered examples.
\begin{itemize}
    \item Section 5 of Droz and Zakharevich's paper \cite{MR4226148} gives a classification of the model structures on a poset, but only up to Quillen equivalence. 
It is a ``classification'' in the sense that a Quillen equivalence class of model structures on a given poset $\mathcal{A}$ is determined by specifying the class of weak equivalences, so if one specifies the sub-posets of $\mathcal{A}$ which are capable of being the weak equivalences of a model structure, then one has specified a list of Quillen-equivalence classes of model structures on $\mathcal{A}$.

A classification of model structures on a poset $\mathcal{A}$ ``on the nose,'' i.e., not up to Quillen equivalence, is obtained from the results of Droz--Zakharevich \cite{MR4226148}, Beke's 2003 preprint \cite{bekenotes}, and Jardine's 2006 paper \cite{MR2252262}, as follows. Given a class $\mathcal{W}$ which is capable of being the class of weak equivalences of a model structure on $\mathcal{A}$, a very special case of Beke's theorem \cite[Proposition 2.1]{bekenotes}, and its dual, shows that there is a unique smallest set $\mathcal{F}_{min}$ (respectively, largest set $\mathcal{F}_{max}$) of maps in $\mathcal{A}$ such that there is a model structure on $\mathcal{A}$ whose fibrations are $\mathcal{F}_{min}$ (resp. $\mathcal{F}_{max}$) and whose weak equivalences are $\mathcal{W}$. Jardine's theorem (attributed to \cite{MR2252262}, but only implicit there; see \cite{nlab:intermediate_model_structure} for a more direct statement and discussion) then establishes that, given any factorization system $(\mathcal{L},\mathcal{R})$ on $\mathcal{A}$ whose right class $\mathcal{R}$ satisfies $\mathcal{F}_{min}\subseteq \mathcal{R}\subseteq\mathcal{F}_{max}$, there exists a (necessarily unique) model structure on $\mathcal{A}$ whose weak equivalences are $\mathcal{W}$ and whose fibrations are $\mathcal{R}$. Hence, by the results of Droz--Zakharevich and Beke and Jardine, each model structure on $\mathcal{A}$ is describable uniquely as a pair $(\mathcal{W},\mathcal{R})$, where $\mathcal{W}$ is a set of morphisms in $\mathcal{A}$ which is capable of extending to a model structure, while $\mathcal{R}$ is the right class of a factorization system on $\mathcal{A}$ satisfying $\mathcal{F}_{min}\subseteq\mathcal{R}\subseteq\mathcal{F}_{max}$.

In this paper, we use some results from Droz--Zakharevich, particularly in \cref{The Droz-Zakharevich and Stanculescu...}. However, our focus in this paper is on a means of classifying model structures on a poset $\mathcal{A}$ in terms of their cofibrant and fibrant objects, rather than the approach yielded by Droz--Zakharevich--Beke--Jardine, which is in terms of the weak equivalences and the fibrations.  For some purposes (i.e., building model structures from topologies and from matroids, as we do in \cref{Model structures constructed from...}), it is much more convenient to describe a model structure by means of cofibrant and fibrant objects than to give a description in terms of weak equivalences. It is also technically more difficult to use a characterization in terms of weak equivalences, since Droz--Zakharevich prove \cite[Theorem A]{MR4226148} that there is no characterization {\em in the language of first-order logic} of which sets of morphisms in $\mathcal{A}$ are the weak equivalences of some model structure on $\mathcal{A}$. By contrast, our Corollary \ref{lettered cor b} {\em is} a first-order characterization of which pairs $(\mathcal{F},\mathcal{C})$ of subsets of $\mathcal{A}$ admit a model structure on $\mathcal{A}$ with $\mathcal{F}$ as fibrant objects and with $\mathcal{C}$ as cofibrant objects.
\item The two preprints \cite{characterizingmodelstructures}, by Mazur--Osorno--Roitzheim--Santhanam--van Niel--Zapata Castro, and \cite{leftandright}, by Carnero Bravo--Goyal--Mart\'{i}nez Alberga--Ng--Roitzheim--Tolosa, are very recent, posted less than a month before this paper. The former gives a classification of model structures which is of the same type as provided by Droz--Zakharevich, Beke, and Jardine---i.e., a classification in terms of pairs $(\mathcal{W},\mathcal{R})$---but with more explicit, usable criteria for determining which sets of morphisms $\mathcal{W}$ and $\mathcal{R}$ satisfy the conditions of Droz--Zakharevich--Beke--Jardine. 
The latter paper, \cite{leftandright}, focuses on identifying minimal sets of morphisms which can be inverted to yield Bousfield (co)localizations of model structures on posets.
Some of the present paper's explicit calculations of factorization systems and of model structures on the $4$-element Boolean algebra also appear in those two preprints. C. Roitzheim suggests (personal communication) that it is worth stressing that the perspective of (co)transfer systems, taken in those two preprints, yields a different methodology and usability than the use of factorization systems.
\item Stanculescu's paper \cite{MR3274498} constructs a model structure with prescribed fibrant objects. The primary thrust of the present paper is that we construct model structures on preorders in which both the fibrant {\em and} the cofibrant objects are prescribed. Hence our results solve a problem which is in one way more difficult than that of \cite{MR3274498} (we solve a problem with more constraints that must be satisfied) but in another way is easier (we work only with preorders, which simplifies the setting). We give a thorough discussion of Stanculescu's main construction from \cite{MR3274498}, and its relationship to our constructions, in \cref{The Droz-Zakharevich and Stanculescu...}.
    \item The paper \cite{MR4467021} of Franchere--Ormsby--Osorno--Qin--Waugh studies the transfer systems on a finite group, which are equivalent to weak factorization systems on the subgroup lattice of the group. The relationship between transfer systems and weak factorization systems on a finite poset is also elucidated in \cite{MR4693627}.  
    \item In the paper \cite{MR4601189} of Balchin--Ormsby--Osorno--Roitzheim, the model structures on a finite linear {\em total} order $[n]$ were enumerated using transfer systems. 

    \item The focus of the paper \cite{MR4667584} of Balchin--MacBrough--Ormsby is on explicitly identifying the model structures on a finite total order $[n]$ among the ``Tamari intervals.'' 
    \item The paper \cite{BCH+25} of Bose--Chih--Housden--Jones--Lewis--Ormsby--Rose deals with the enumeration of factorization systems rather than model structures. There is significant overlap between \cite{BCH+25} and the review material we present in \cref{All strong model structures...}, but our new results in this paper do not overlap with results of \cite{BCH+25}. 
    \item To avoid any potential confusion: the paper \cite{MR2721035} of Raptis is about model structures on the category of posets, which is an entirely different matter from a model structure on an {\em individual} poset (or preorder), which is the subject of this paper.
\end{itemize}

\subsection{Acknowledgments}

The first author thanks I. Zakharevich for a conversation about the paper \cite{MR4226148}, J. Beardsley for discussions on matroids, and for suggesting the book \cite{MR1783451}, and C. Roitzheim for a discussion about the preprints \cite{leftandright} and \cite{characterizingmodelstructures}. The second author thanks Joey Beauvais–Feisthauer for helpful suggestions regarding the Sage script used to count the strong model structures. We used Wayne State University's high-performance computing grid for some calculations described in this paper.

\subsection{Conventions}

\begin{itemize}
\item Given a class of morphisms $S$ in some category $\mathcal{A}$, we follow Hovey \cite{MR1650134} in writing $\Proj(S)$ (respectively, $\Inj(S)$) for the class of morphisms with the left (respectively, right) lifting property with respect to $S$.
\item There are no set-theoretic tricks or surprises in this paper: all our preorders are preorders in the standard sense in mathematics. In particular, a preorder has only a {\em set} of objects.
\item In this paper, we do not assume that our model categories have {\em all} (small) limits and colimits. Rather, we adopt Quillen's original convention \cite{MR0223432} that a model category is a category with {\em finite} limits and colimits which is equipped with a model structure.
\item Breaking with Quillen's conventions from \cite{MR0223432}, we adhere to the modern convention that all model structures are assumed to be closed, i.e., the pairs $(cof, we\cap fib)$ and $(cof\cap we, fib)$ are each weak factorization systems\footnote{Recall that a \emph{weak factorization system} on a category $\mathcal{A}$ consists of a pair $(\mathcal{L}, \mathcal{R})$ of full subcategories of $\mathcal{A}$ such that every morphism in $\mathcal{A}$ can be factored as a map in $\mathcal{L}$ followed by a map in $\mathcal{R}$, such that $\mathcal{L} = \Proj(\mathcal{R})$, and such that $\mathcal{R} = \Inj(\mathcal{L})$.}.
\item In this paper, a {\em model structure} on a category $\mathcal{A}$ consists of an ordered triple $(cof, we, fib)$ of retract-closed full subcategories $\mathcal{A}$ such that $(cof \cap we, fib)$ and $(cof,we \cap fib)$ are each weak factorization systems, and such that $we$ has the two-out-of-three property. 

A model category, by our stated conventions, is consequently a finitely bicomplete category equipped with a model structure. To be clear, if we only write that ``$\mathcal{A}$ is a category equipped with a model structure,'' we are {\em not} asserting that $\mathcal{A}$ has any particular limits or colimits at all. Hence ``category with model structure'' is weaker than ``model category.''
\item In this paper, we will very frequently be concerned with model structures on preorders which do not necessarily have all finite limits or colimits, but which at least have an initial object and a final object. It will be convenient to have a short phrase for such things. By a ``model preorder'' we will mean a preorder with an initial object and a final object, and which is equipped with a model structure.
\end{itemize}

\section{Basic observations about model structures on preorders}
\label{basic properties}

All the results in this section are well-known to those who have written about model structures on preorders,
and the results are each proven by short, routine arguments, so we omit most of the proofs. A few of these results also have known generalizations from preorders to categories enriched in posets, by \cite{MR4140153}.

\subsection{Every model preorder yields a pair of (co)reflective subcategories}

\begin{prop}\label{factorization systems are functorial}
Let $\mathcal{A}$ be a preorder. Then every weak factorization system for a closed model structure on $\mathcal{A}$ is an orthogonal factorization system.
Furthermore, the factorizations in any closed model structure on $\mathcal{A}$ are uniquely determined, up to isomorphism. Consequently, in any model structure on a preorder, the factorization systems are functorial.
\end{prop}
In light of Proposition \ref{factorization systems are functorial}, there is no difference between weak factorization systems and orthogonal factorization systems on preorders. Consequently, whenever there is no risk of ambiguity, we will use the phrase ``factorization system'' to mean both ``weak factorization system'' and ``orthogonal factorization system.''

\begin{corollary}\label{retractile and sectile cor}
    Let $\mathcal{A}$ be a preorder equipped with a factorization system $(\mathcal{L},\mathcal{R})$. Then $X \stackrel{f}{\longrightarrow} Y$ is in $\mathcal{R}$ if and only if, when $f$ is factored as a map in $\mathcal{L}$ followed by a map in $\mathcal{R}$, the map in $\mathcal{L}$ is an isomorphism.
The dual statement holds as well.
\end{corollary}

Definition \ref{retractile and sectile classes} is not standard, but it is also not new: the terms ``retractile'' and ``sectile'' were introduced, with the same meanings, in \cite{MR3775351}.
\begin{definition}\label{retractile and sectile classes}
    A class $\mathcal{R}$ of morphisms in a category is {\em retractile} if, whenever $g\circ f$ is in $\mathcal{R}$, the morphism $f$ is also in $\mathcal{R}$. Dually, a class $\mathcal{S}$ of morphisms is {\em sectile} if, whenever $g\circ f$ is in $\mathcal{S}$, the morphism $g$ is also in $\mathcal{S}$.
\end{definition}

\begin{prop}\label{fibrations are retractile}
In any factorization system $(\mathcal{L},\mathcal{R})$ on a preorder, $\mathcal{R}$ is retractile. Dually, $\mathcal{L}$ is sectile.
\end{prop}

Given a factorization system $(\mathcal{L},\mathcal{R})$ on a preorder, 
Proposition \ref{fibrations are retractile} establishes that $\mathcal{L}$ is sectile. It may or may not be the case, however, that $\mathcal{L}$ is also retractile. It will be useful (e.g. in Proposition \ref{sectile fibs prop}, and Proposition \ref{strong model str satisfies strong compatibility}) to have a name for factorization systems with such properties. We will call such factorization systems ``retractile,'' as follows:
\begin{definition}\label{def of retractile fs}
Let $(\mathcal{L},\mathcal{R})$ be a factorization system on a category. We will say that $(\mathcal{L},\mathcal{R})$ is {\em retractile} (respectively, {\em sectile}) if $\mathcal{L}$ is retractile (respectively, $\mathcal{R}$ is sectile).
\end{definition}

\begin{remark} 
From the original topological motivations for model categories, we are used to the idea that fibrations ought to behave something like surjections, like how Serre fibrations and Hurewicz fibrations are surjective. Surjections (or more generally, epimorphisms), however, are generally not retractile; instead, it is the {\em monomorphisms} in any category that are retractile. So we see that our topological intuition about model structures does not serve us perfectly well when considering model structures on preorders.

The category of sets admits a model structure in which the cofibrations are the surjections, all maps are weak equivalences, and the fibrations are the injections. In this model structure, the fibrations are retractile. One moral of Proposition \ref{fibrations are retractile} is that model structures on preorders behave, in a particular sense, like this model structure on the category of sets, rather than behaving like the Serre model structure on topological spaces.
\end{remark}

\begin{notation}
    Suppose we are given a preorder $\mathcal{A}$ equipped with a model structure. For any morphism $X \xrightarrow{f} Y$ in $\mathcal{A}$ , we will write $\Phi(X,Y)$ for the element of $\mathcal{A}$, unique up to isomorphism, such that $X \rightarrow \Phi(X,Y) \rightarrow Y$ is the acyclic-cofibration-followed-by-fibration factorization of $f$. Furthermore, we will write $\Psi(X,Y)$ for the element of $\mathcal{A}$, unique up to isomorphism, such that $X \rightarrow \Psi(X,Y) \rightarrow Y$ is the cofibration-followed-by-acyclic-fibration factorization of $f$.
\end{notation}

\begin{prop}\label{identification of replacements}
Let $\mathcal{A}$ be a preorder equipped with a model structure.
\begin{itemize}
\item Suppose that $\mathcal{A}$ has an initial object $0$. Then, for each cofibrant object $X$ of $\mathcal{A}$ and each morphism $f: X\rightarrow Y$ in $\mathcal{A}$, the object 
$\Psi(X,Y)$ is isomorphic to $\Psi(0,Y)$. Furthermore, $\Psi(0,0)$ is initial.
\item
Dually, suppose that $\mathcal{A}$ has a terminal object $1$. Then, for each fibrant object $Y$ of $\mathcal{A}$ and each morphism $f: X\rightarrow Y$ in $\mathcal{A}$, the object 
$\Phi(X,Y)$ is isomorphic to $\Phi(X,1)$.  Furthermore, $\Phi(1,1)$ is terminal.
\end{itemize}
\end{prop}

The object $\Psi(0,X)$ is the {\em cofibrant replacement of $X$}, while $\Phi(X,1)$ is the {\em fibrant replacement of $X$}. In light of Proposition \ref{factorization systems are functorial}, cofibrant replacement is functorial and well-defined up to isomorphism, and similarly for fibrant replacement. Proposition \ref{monadic replacements} goes further, and shows that cofibrant and fibrant replacements are not only functors, but are in fact (co){\em monads}:

\begin{prop}\label{monadic replacements}
Let $\mathcal{A}$ be a model preorder, i.e., a preorder with initial and terminal object and which is equipped with a model structure. Then the fibrant replacement functor $F: \mathcal{A}\rightarrow\mathcal{A}$ on $\mathcal{A}$ has the structure of a monad; any two choices of monad structure on $F$ are isomorphic; and any monad structure on $F$ is idempotent. 
\end{prop}

\begin{corollary}\label{identification of fibrant objects}
Let $\mathcal{A}$ be a model preorder. 
An object $X$ of $\mathcal{A}$ is fibrant if and only if 
the unit map $\eta(X) : X\rightarrow FX$ of the fibrant replacement monad is an isomorphism. 
\end{corollary}

\begin{prop}\label{cofs between cofibrants}
Let $\mathcal{A}$ be a model preorder, 
and let $f:X \rightarrow Y$ be a morphism in $\mathcal{A}$. 
\begin{enumerate}
\item If $Y$ is fibrant, then $f: X \rightarrow Y$ is a fibration if and only if $X$ is fibrant.
\item If $X$ and $Y$ are each bifibrant (that is, both cofibrant and fibrant), 
then $f$ is a weak equivalence if and only if $f$ is an isomorphism.
\end{enumerate}
\end{prop}

Here is the upshot of Propositions \ref{monadic replacements} and \ref{cofs between cofibrants} and Corollary \ref{identification of fibrant objects} (and their duals, which are also of course true): given a model preorder $\mathcal{A}$, one gets:
\begin{itemize}
\item an idempotent monad $F$ (fibrant replacement) on $\mathcal{A}$, unique up to isomorphism.
\item Consequently one gets a reflective replete\footnote{``Replete'' means ``closed under isomorphisms.''} subcategory $\Fib(\mathcal{A})$ of $\mathcal{A}$ consisting of the fibrant objects.
\item Dually, one gets an idempotent comonad $C$ (cofibrant replacement) on $\mathcal{A}$, also unique up to isomorphism, and consequently a coreflective replete subcategory $\Cofib(\mathcal{A})$ of $\mathcal{A}$ consisting of the cofibrant objects.
\item Finally, the homotopy category of the given model structure is precisely the intersection $\Fib(\mathcal{A})\cap \Cofib(\mathcal{A}) \subseteq\mathcal{A}$.
\end{itemize}

\section{Which pairs of (co)reflective subcategories arise from model preorders?}
\label{model structures from compatible pairs}

In the previous subsection, we explained how a model preorder $\mathcal{A}$ yields a pair $(\Fib(\mathcal{A}),\Cofib(\mathcal{A}))$, with $\Fib(\mathcal{A})$ a reflective subcategory of $\mathcal{A}$, and with $\Cofib(\mathcal{A})$ a coreflective subcategory of $\mathcal{A}$. In this section, we study the question of whether one can go in the other direction, i.e., given a preorder $\mathcal{A}$ with initial object and final object, a reflective subcategory $\Fib$ of $\mathcal{A}$, and a coreflective subcategory $\Cofib$ of $\mathcal{A}$, whether we can find a model structure on $\mathcal{A}$ such that $\Fib(\mathcal{A}) = \Fib$ and $\Cofib(\mathcal{A}) = \Cofib$. In 
Theorem \ref{characterization thm}, we show that the answer is ``yes'' if and only if $\Fib,\Cofib$ satisfy two simple conditions which make them a ``compatible pair'' (Definition \ref{def of compatible pair}).

\subsection{Compatible pairs}
We introduce the notion of ``compatible pairs.'' The results in this subsection establish basic properties of compatible pairs, and are quite straightforward to prove, so we leave off the proofs. 

\begin{prop}\label{compatibility conditions between comonad and monad}
Let $\mathcal{A}$ be a model preorder. Then the following conditions are all true:
\begin{enumerate}
\item Let $F,C$ be the fibrant replacement functor and cofibrant replacement functor, respectively, on $\mathcal{A}$. Then $F$ and $C$ commute up to isomorphism. 
\item If $X,Y$ are objects of $\mathcal{A}$ such that $CX\leq FY$, then $CX\leq FCY$ and $CFX\leq FY$.
\item If $X,Y$ are objects of $\mathcal{A}$ such that $X\leq CY$ and $FX\leq Y$, then $FX\leq CY$.
\end{enumerate}
\end{prop}

\begin{remark}
If we are given a preorder $\mathcal{A}$, and rather than a model structure on $\mathcal{A}$, we are given an idempotent monad $F$ on $\mathcal{A}$ and an idempotent comonad $C$ on $\mathcal{A}$, then we have the following implications between the conditions from Proposition \ref{compatibility conditions between comonad and monad}: conditions 2 and 3 jointly imply condition 1, and condition 1 implies condition 2. 
\end{remark}

%

\begin{prop}\label{weak equivs}
Let $\mathcal{A}$ be a model preorder, and let $f: V\rightarrow W$ be a morphism in $\mathcal{A}$. Then the following conditions are equivalent:
\begin{enumerate}
\item $f$ is a weak equivalence.
\item $CFf: CFV \rightarrow CFW$ is an isomorphism.
\item $FCf: FCV \rightarrow FCW$ is an isomorphism.
\end{enumerate}
\end{prop}
\begin{prop}\label{three conditions}
Let $\mathcal{A}$ be a preorder with initial and terminal object and which is equipped with a model structure, and let $f: X\rightarrow Y$ be a morphism in $\mathcal{A}$. 
Then, among the following conditions, the first implies the latter two, and the latter two are equivalent:
\begin{enumerate}
\item $f$ is an acyclic fibration.
\item The counit map $\epsilon(Y): CY \rightarrow Y$ factors as an acyclic fibration $CY\rightarrow X$ followed by $f$.
\item $Cf$ is an isomorphism. 
\end{enumerate}
\end{prop}

In Proposition \ref{sectile fibs prop} and throughout the rest of this subsection, we continue to write $F$ for the fibrant replacement monad, and $C$ for the cofibrant replacement comonad, on a model preorder.
\begin{prop}\label{sectile fibs prop}
Let $\mathcal{A}$ be a model preorder. Then $\mathcal{A}$ has sectile acyclic fibrations if and only if, for every map $f$ in $\mathcal{A}$ such that $Cf$ is an isomorphism, the map $f$ is an acyclic fibration. In other words: $\mathcal{A}$ has sectile acyclic fibrations if and only if the three conditions of Proposition \ref{three conditions} are all equivalent in $\mathcal{A}$.

Dually, $\mathcal{A}$ has retractile acyclic cofibrations if and only if, for every map $f$ in $\mathcal{A}$ such that $Ff$ is an isomorphism, the map $f$ is an acyclic cofibration.
\end{prop}

\begin{definition}\label{def of strong model structures}
A model structure is {\em strong} if its acyclic fibrations are sectile and its acyclic cofibrations are retractile. 
\end{definition}

\begin{prop}\label{strong model str satisfies strong compatibility}
Let $\mathcal{A}$ be a strong model preorder. If $f:X \rightarrow Y$ is inverted by $C$ and also inverted by $F$, then $f$ is an isomorphism.
\end{prop}

\begin{definition}\label{def of compatible pair}
Let $\mathcal{A}$ be a preorder. By a {\em compatible pair on $\mathcal{A}$} we mean a monad $F$ on $\mathcal{A}$ and a comonad $C$ on $\mathcal{A}$ such that the first and third conditions (equivalently, all the conditions) from Proposition \ref{compatibility conditions between comonad and monad} are satisfied. That is:
\begin{itemize}
\item $F$ commutes with $C$ up to isomorphism, and
\item if $X,Y$ are objects of $\mathcal{A}$ such that $X\leq CY$ and $FX\leq Y$, then $FX\leq CY$.
\end{itemize}

We say that a compatible pair $F,C$ is a {\em strongly compatible pair} if it is a compatible pair satisfying:
\begin{itemize}
\item if $f:X \rightarrow Y$ is a morphism in $\mathcal{A}$ which is inverted by $C$ and also inverted by $F$, then $f$ is an isomorphism.
\end{itemize}
\end{definition}

\begin{prop}\label{compatible pairs from model preorders}
Given a model preorder $\mathcal{A}$, let $F$ be the fibrant replacement monad on $\mathcal{A}$, and let $C$ be the cofibrant replacement comonad on $\mathcal{A}$, as in Proposition \ref{monadic replacements}. Then $(F,C)$ is a compatible pair. If the model structure on $\mathcal{A}$ is strong, then $(F,C)$ is a strongly compatible pair.
\end{prop}

\subsection{The Droz-Zakharevich and Stanculescu model structures arising from a compatible pair}
\label{The Droz-Zakharevich and Stanculescu...}

Given a compatible pair $(F,C)$ on a bicomplete preorder $\mathcal{A}$, we aim to produce a model structure on $\mathcal{A}$ in which $F$ is the fibrant replacement monad, and $C$ is the cofibrant replacement comonad. We will go about this in a gradual way, by first considering two constructions already in the literature---one due to Droz--Zakharevich \cite{MR4226148}, and the other due to Stanculescu \cite{MR3274498}---which are natural guesses for such a model structure on $\mathcal{A}$. We will see that each of these model structures will do the job as long as $(F,C)$ is {\em strongly} compatible. However, for a general compatible (not necessarily strongly compatible) pair $(F,C)$, the Droz-Zakharevich (DZ) and Stanculescu model structures each have the correct weak equivalences, but too many cofibrations. In the next subsection, we will modify the Stanculescu model structure, cutting down the class of cofibrations, to produce a new model structure on $\mathcal{A}$ which is of precisely the desired kind.

First we consider the DZ model structure. Droz and Zakharevich, in \cite{MR4226148}, provide a method (Proposition \ref{dz method}, below) to construct a model structure with a fixed class of weak equivalences, using a ``decomposable class'' and a ``center.'' 
Here are the relevant definitions from \cite{MR4226148}:
\begin{definition}[\cite{MR4226148}]\label{definition of a center}
Let $\mathcal{A}$ be a preorder.
\begin{itemize}
\item
    A subcategory $\mathcal{E}$ of $\mathcal{A}$ is \emph{decomposable}\footnote{The motivating example of a decomposable subcategory is the class of all weak equivalences in a model preorder. Another class of examples, in light of Proposition \ref{fibrations are retractile}: given a preorder equipped with a factorization system $(\mathcal{L},\mathcal{R})$, the left class $\mathcal{L}$ is retractile if and only if it is decomposable. Similarly, the right class $\mathcal{R}$ is sectile if and only if it is decomposable.} if, for any morphism $f \in \mathcal{E}$ such that $f = hg$ for some morphisms $g$ and $h$, both $g$ and $h$ are in $\mathcal{E}$. 
In other words, a subcategory $\mathcal{E}$ is decomposable if and only if it is both sectile and retractile (Definition \ref{retractile and sectile classes}).
\item
Given a decomposable subcategory $\mathcal{W}$ of $\mathcal{A}$, a \emph{center associated to $\mathcal{W}$} is a functor
\[
\chi \colon \mathcal{A} \longrightarrow \mathcal{A}
\]
such that the following properties hold:
\begin{itemize}
    \item[\textbf{C1:}] The image of $\chi|_{\mathcal{W}}$ only contains isomorphisms.
    \item[\textbf{C2:}] For all $A \in \mathcal{A}$, the diagram
    \[
    \begin{tikzcd}
    A \times \chi(A) \arrow[d] \arrow[r] & \chi(A) \arrow[d] \\
    A \arrow[r] & A \cup \chi(A)
    \end{tikzcd}
    \]
    lies in $\mathcal{W}$.
\end{itemize}   
\end{itemize}
\end{definition}

\begin{definition}\label{def of semifib}
Let $\mathcal{A}$ be a preorder with initial and terminal objects.
Given a compatible pair $(F,C)$ on $\mathcal{A}$, we adopt the following terminology:
\begin{itemize}
\item an element $X$ of $\mathcal{A}$ is {\em $C$-cofibrant} (respectively, {\em $F$-fibrant}) if the map $CX \rightarrow X$ (respectively, $X\rightarrow FX$) is an isomorphism.
\item A map $X \rightarrow Y$ is {\em $C$-anodyne} if it has the right lifting property with respect to all maps with $C$-cofibrant codomain. Equivalently: $X \rightarrow Y$ is $C$-anodyne if, for each element $T$ of $\mathcal{A}$ such that $CT\leq Y$, we also have $CT\leq X$.
\item An element $X$ of $\mathcal{A}$ is {\em semifibrant} (respectively, {\em semicofibrant}) if $FCX\leq X$ (respectively, $X\leq FCX$).
\end{itemize}
\end{definition}
Given a center $\chi$ on a poset $\mathcal{A}$, Droz--Zakharevich \cite{MR4226148} say that an element $X$ of a poset is ``semifibrant'' (respectively, ``semicofibrant'') if $\chi (X) \leq X$ (respectively, $X\leq \chi(X)$). We will show in Proposition \ref{CP generate model structures} that, given a compatible pair $(F,C)$ on a preorder $\mathcal{A}$, the composite $FC$ is a center on $\mathcal{A}$. Hence our usage of the terms ``semifibrant'' and ``semicofibrant'' are special cases of Droz--Zakharevich's use.

\begin{definition}\label{def of I}
    Given a functor $\mathcal{F}: \mathcal{A} \rightarrow \mathcal{B}$, we let $\mathcal{I}_{\mathcal{F}}$ be the class of morphisms in $\mathcal{A}$ which are inverted by $\mathcal{F}$, i.e., 
\begin{align*}
 \mathcal{I}_{\mathcal{F}} &= \left\{f \in \Mor(\mathcal{A}) \mid Ff \, \text{is an isomorphism} \right\}.
\end{align*}
\end{definition}

\begin{prop}\label{semifibrancy characterization}
With definitions as in Definition \ref{def of semifib}, given an element $X$ of $\mathcal{A}$, the following conditions are equivalent:
\begin{enumerate}
\item $X$ is semifibrant.
\item $CX$ is $F$-fibrant.
\item The map $X\rightarrow FX$ is in $\mathcal{I}_C$.
\end{enumerate}
\end{prop}
\begin{proof}
If $X$ is semifibrant, then $FCX\leq X$, hence $CX\leq FCX \cong CFCX \leq CX$, i.e., $CX$ is $F$-fibrant. The converse, and the equivalence with the third condition, are even easier.
\end{proof}

The following proposition is Proposition 3.8 from \cite{MR4226148}.
\begin{prop} \label{dz method}
    Given a decomposable subcategory $\mathcal{W}$ of a finitely bicomplete preorder $\mathcal{A}$, and given a center $\chi$ associated to $\mathcal{W}$, there exists a model structure on $\mathcal{A}$ whose class of weak equivalences is precisely $\mathcal{W}$, and whose acyclic fibrations are the maps with the right lifting property with respect to all maps with semicofibrant codomain. 
\end{prop}
Given $\mathcal{A},\mathcal{W},\chi$ as in Proposition \ref{dz method}, we will call the model structure given by Proposition \ref{dz method} the {\em DZ model structure}. It is easy to see that there is also a ``dual DZ model structure'' on $\mathcal{A}$ with the same weak equivalences, but whose acyclic cofibrations are the maps with the left lifting property with respect to all maps with semifibrant domain. However, as Droz and Zakhrevich point out, their construction is {\em not} self-dual: the cofibrations in the dual DZ model structure do not usually agree with the cofibrations in the DZ model structure, and similarly for fibrations. 

\begin{prop}\label{CP generate model structures}
    Given a compatible pair $(F,C)$ on a finitely bicomplete preorder $\mathcal{A}$, the composite functor $FC: \mathcal{A}\rightarrow\mathcal{A}$ is a center associated to $\mathcal{I}_{FC}$. Consequently there exists a DZ model structure on $\mathcal{A}$ whose set of weak equivalences is $\mathcal{I}_{FC}$, and in which every semicofibrant object is cofibrant.
\end{prop}
\begin{proof}
It is clear from the definition of the DZ model structure that, if it exists, then the weak equivalences are precisely $\mathcal{I}_{FC}$, and the semicofibrant objects are cofibrant. Hence our task is to show existence of the DZ model structure.
It follows easily from Proposition \ref{weak equivs} that $\mathcal{W}$ is decomposable, so by Proposition \ref{dz method}, all we need is to show that the functor $\chi: \mathcal{A} \rightarrow \mathcal{A}$ defined by $\chi(A) = FC(A)$ is a center associated to $\mathcal{W} = \mathcal{I}_{FC}$.

    Condition C1 in Definition \ref{definition of a center} is satisfied by $\chi$ simply because of how $\mathcal{W}$ is defined.
    All we need now is to show that for any object $A \in \mathcal{A}$, the diagram 
    \begin{center}
        \begin{tikzcd}
            A \times FC(A) \arrow[r] \arrow[d] & FC(A) \arrow[d] \\
            A \arrow[r] & A \cup FC(A)
        \end{tikzcd}
    \end{center}
    lies in $\mathcal{W}$, that is, the diagram 
    \begin{center}
    \begin{equation}\label{center equation}
        \begin{tikzcd}
            FC(A \wedge FC(A)) \arrow[r] \arrow[d] & FC(A) \arrow[d] \\
            FC(A) \arrow[r] & FC(A \vee FC(A))
        \end{tikzcd}
    \end{equation}    
    \end{center}
    must have all the maps as isomorphisms.
    For that purpose, consider the special case of the second axiom of Definition \ref{def of compatible pair} in which $X$ is $CA$, and $Y$ is $F(A \wedge FC(A))$. The universal property of the product (i.e., the meet) in $\mathcal{A}$ yields the dotted arrow in the diagram
    \begin{center}
        \begin{tikzcd}
C(A) \arrow[rr] \arrow[dd] \arrow[rd, dashed] &                                      & FC(A) \arrow[dd] \\
                                              & A \wedge FC(A) \arrow[ru] \arrow[ld] &                  \\
A \arrow[rr]                                  &                                      & F(A)            .
\end{tikzcd}
    \end{center}
Hence $C(A) \leq A \wedge FC(A) \leq F(A \wedge FC(A))$. Hence $C(A) \cong CC(A) \leq CF(A \wedge FC(A))$, i.e., $X \leq CY$. Applying the functor $F$ to the map $CA \rightarrow A \wedge FC(A)$ yields that $FX \leq Y$. Hence, by the first and second axioms in Definition \ref{def of compatible pair}, we have $FC(A) \leq FC(A \wedge FC(A))$. Since we are in a preorder, the top horizontal map in \ref{center equation} is thus an isomorphism.

We now aim to show that the right-hand vertical map in \eqref{center equation} is also an isomorphism in $\mathcal{A}$. By the universal property of the coproduct (i.e., the join), we have a morphism filling in the dotted arrow in the commutative diagram
    \begin{center}
        \begin{tikzcd}
C(A) \arrow[rr] \arrow[dd] &                                 & FC(A) \arrow[dd] \arrow[ld] \\
                           & A \vee FC(A) \arrow[rd, dashed] &                             \\
A \arrow[rr] \arrow[ru]    &                                 & F(A).                       
\end{tikzcd}
    \end{center}
We apply $F$ to the resulting inequalities \begin{equation}\label{ineqs 22} C(A \vee FC(A)) \leq A \vee FC(A) \leq F(A),\end{equation} obtaining $FC(A \vee FC(A)) \leq FF(A) \cong F(A)$. Now consider the case of the second axiom in Definition \ref{def of compatible pair} in which $X = C(A \vee FC(A))$ and $Y = F(A)$. We have just shown that $FX \leq Y$. Applying $C$ to \eqref{ineqs 22} yields that $C(A \vee FC(A)) \cong CC(A \vee FC(A)) \leq CF(A)$, i.e., $X \leq CY$. The first and second axioms of Definition \ref{def of compatible pair} consequently give us that $FC(A \vee FC(A)) \leq FC(A)$, i.e., the right-hand vertical map in \ref{center equation} is an isomorphism. 

We have shown that the top horizontal and right-hand vertical maps in the square \eqref{center equation} are isomorphisms in the preorder $\mathcal{A}$. In a preorder, the class of all isomorphisms is a decomposable subcategory, hence all maps in \eqref{center equation} are isomorphisms, as desired.
\end{proof}

Given a compatible pair $(F,C)$ on a preorder $\mathcal{A}$, Proposition \ref{semifibrancy characterization}, establishes that every $C$-cofibrant object is semicofibrant. If $(F,C)$ is not {\em strongly} compatible, then it may happen that not every semicofibrant is $C$-cofibrant. For example, on the power set of a two-element set, precisely three out of the 23 model structures---namely, $(4,4)$ and $(6,6)$ and $(8,8)$ from \eqref{n2 model structs}---have at least one semicofibrant object which is not $C$-cofibrant\footnote{From inspection of \eqref{n2 model structs}, one observes that in each of the model structures on $P(\{a,b\})$, either all the semicofibrants are cofibrant, or all the semifibrants are fibrant, or both. The smallest example of a model preorder in which there are {\em both} non-fibrant semifibrants and non-cofibrant semicofibrants is the following model structure on the total order with $5$ elements:
\[\xymatrix{
 \bullet \ar[r]^f \ar@/_/[rr]_c & \bullet \ar[r]^c & \bullet \ar[r]^f\ar@/_/[rr]_f & \bullet \ar[r]^c & \bullet
}\]
where all maps are weak equivalences, the maps marked $c$ and $f$ are cofibrations and fibrations respectively, and the maps not drawn are neither cofibrations nor fibrations.}. 
Hence the DZ model structure doesn't solve our problem: its cofibrants are not precisely the $C$-cofibrant objects.

Next we consider the Stanculescu model structure, defined in \cite{MR3274498}:
\begin{definition}\label{def of stanculescu model struct}
Given a compatible pair $(F,C)$ on a bicomplete preorder $\mathcal{A}$, the {\em Stanculescu model structure} is the model structure on $\mathcal{A}$ defined as follows:
\begin{itemize}
\item The set of weak equivalence is the set $\mathcal{I}_{FC}$ of morphisms inverted by $FC$. 
\item The set of acyclic cofibrations is $\mathcal{I}_F$.
\item Consequently the set of fibrations is the set $\Inj(\mathcal{I}_F)$ of maps with the right lifting property with respect to $\mathcal{I}_F$.
\item Consequently the set of cofibrations is the set $\Proj(\mathcal{I}_{FC}\cap \Inj(\mathcal{I}_F))$ of maps with the left lifting property with respect to $\mathcal{I}_{FC}\cap \Inj(\mathcal{I}_F)$. 
\end{itemize}
\end{definition}
Stanculescu's construction is more general, applying to bicomplete categories, not just preorders. 

Stanculescu shows that the fibrant objects in the Stanculescu model structure are precisely the $F$-fibrant objects. Indeed, this is Stanculescu's motivation: it is why the title of Stanculescu's paper about this model structure, \cite{MR3274498}, is ``Constructing model categories with prescribed fibrant objects.'' This property of the Stanculescu model structure is a clear step forward for us, since are trying to produce a model structure on $\mathcal{A}$ in which the fibrant objects are the $F$-fibrants---but also in which {\em the cofibrant objects are the $C$-cofibrants}. One can easily show that, if there is such a model structure on $\mathcal{A}$, then its weak equivalences must be $\mathcal{I}_{FC}$. Hence the DZ model structure gets the weak equivalences right, and the Stanculescu model structure gets both the weak equivalences and the fibrants right.

However, the Stanculescu model structure still has too many cofibrants to solve our problem:
\begin{prop}\label{stanculescu cofibrants}
The cofibrant elements in the Stanculescu model structure on $\mathcal{A}$ are precisely the semicofibrant elements.
\end{prop}
\begin{proof}
Consider a lifting problem
\[\xymatrix{
 0 \ar[d] \ar[r] & S \ar[d]^f \\ 
 X \ar[r] & T 
}\]
with $X$ semicofibrant, and with $f$ a Stanculescu acyclic fibration, i.e., $f\in \mathcal{I}_{FC}\cap \Inj(\mathcal{I}_F)$. Since $CS \rightarrow CT$ is in $\mathcal{I}_F$, it lifts over $S\rightarrow T$, i.e., $CT\leq S$. We now have $CX\leq CT \leq S$, hence a square
\begin{equation}\label{lifting problem 35851332}\xymatrix{
 CX \ar[r] \ar[d] & S \ar[d] \\ X\ar[r] & T
}\end{equation}
whose right-hand vertical map is in $\Inj(\mathcal{I}_F)$. Since $X$ is semicofibrant, the left-hand vertical map is in $\mathcal{I}_F$, by the dual of Proposition \ref{semifibrancy characterization}. Hence the lifting problem \eqref{lifting problem 35851332} is solvable, i.e., $X \leq S$, solving the original lifting problem as well. Hence every semicofibrant element of $\mathcal{A}$ is also a Stanculescu cofibrant.

For the converse: for an arbitrary element $X$ of $\mathcal{A}$, it is easy to check that $FCX\rightarrow X$ is in $\Inj(\mathcal{I}_F)$ and in $\mathcal{I}_{FC}$, hence also in $\mathcal{I}_{FC}\cap \Inj(\mathcal{I}_F)$, i.e., it is a Stanculescu acyclic fibration. Hence, if $X$ is a Stanculescu cofibrant, then the lifting problem
\[\xymatrix{
 0 \ar[d] \ar[r] & FCX \ar[d] \\ 
 X \ar[r] & FX 
}\]
has a solution, i.e., $X$ is semicofibrant.
\end{proof}
In the discussion preceding Definition \ref{def of stanculescu model struct}, we observed that for compatible pairs which are not strongly compatible, the $C$-cofibrants may fail to agree with the semicofibrant objects. Hence Proposition \ref{stanculescu cofibrants} shows that the Stanculescu model structure is not our desired model structure in which the cofibrants are the $C$-cofibrants and the fibrants are the $F$-fibrants.

\subsection{A model structure with prescribed fibrants {\em and} cofibrants}

One would ideally like to have a model structure on $\mathcal{A}$ with the same weak equivalences as the Stanculescu model structure, and whose cofibrations are those Stanculescu cofibrations which {\em aren't} of the form $0\rightarrow X$ with $X$ non-$C$-cofibrant. In other words, one would like to simply remove the undesired cofibrations from the Stanculescu model structure, and to thereby obtain the set of cofibrations of some other model structure, without changing the weak equivalences and without changing the fibrant objects, so that the cofibrant objects in the resulting model structure are precisely the $C$-cofibrants. Of course one does not expect this to work so simply, since removing maps from some class of maps usually is not a well-behaved operation at all.

The surprise is that this na\"{i}ve idea simply works, exactly as stated. Let $\mathcal{Z}$ denote the class of maps
\begin{align*}
 \mathcal{Z} &= \left\{ X\stackrel{f}{\longrightarrow} Y\in \Proj\left(\mathcal{I}_{FC}\cap \Inj(\mathcal{I}_F)\right) : X\ncong 0,\mbox{\ or\ } Y\mbox{\ is\ } C\mbox{-cofibrant}\right\}.
\end{align*}
Recall that, given a cocomplete category $\mathcal{A}$, a class of morphisms $\mathcal{M}$ in $\mathcal{B}$ is called {\em cofibrantly closed} if $\mathcal{M}$ is closed under retracts, composites and transfinite composites, contains all isomorphisms, and is stable under pushouts. For any weak factorization system $(\mathcal{L},\mathcal{R})$ on $\mathcal{A}$, the left class $\mathcal{L}$ is cofibrantly closed. 

A cofibrantly closed class of morphisms $\mathcal{M}$ is {\em cofibrantly generated} if there exists a {\em set} of morphisms $I$ of $\mathcal{A}$ such that $\mathcal{M}$ is the smallest cofibrantly closed class of morphisms in $\mathcal{M}$ containing $I$. Given a cofibrantly generated class $\mathcal{M}$ in a locally presentable category $\mathcal{A}$, Quillen's small object argument \cite{MR0223432} constructs a weak factorization system on $\mathcal{A}$ whose left class is precisely $\mathcal{M}$. See \cite[section 3]{MR1916156} for a nice brief discussion, \cite[Theorem 1.3]{rosickypreprint} for a direct statement of the relevant result, or \cite{MR2506256} for a paper-length development of the ideas.

In the case that the category $\mathcal{A}$ is a preorder, the smallness of $\mathcal{A}$ ensures both that $\mathcal{A}$ is locally presentable and that {\em every} cofibrantly closed class of morphisms in $\mathcal{A}$ is cofibrantly generated. The point is that, if we can show that $\mathcal{Z}$ is cofibrantly closed, then $\mathcal{Z}$ is the left class of some factorization system on the bicomplete preorder $\mathcal{A}$, so $\mathcal{Z}$ can plausibly be the class of cofibrations of some model structure on $\mathcal{A}$. But it is elementary to check that $\mathcal{Z}$ is indeed a cofibrantly closed class: it amounts to checking that $\Proj\left(\mathcal{I}_{FC}\cap \Inj(\mathcal{I}_F)\right)$ is a cofibrantly closed class (which we know, since is the class of cofibrations in the Stanculescu model structure), and that the class of $C$-cofibrant objects of $\mathcal{C}$ is closed under retracts (automatic since $\mathcal{A}$ is a preorder), pushouts and ordinal-indexed colimits (both satisfied since coreflective subcategories are closed under colimits). This will be the key to constructing the desired model structure on $\mathcal{A}$ in Theorem \ref{main construction thm}. We first need a lemma:
\begin{lemma}\label{lifting lemma}
Let $\mathcal{A}$ be a finitely complete preorder, and let $G: \mathcal{A}\rightarrow\mathcal{A}$ be a functor. Let $\mathcal{M}$ be a set of morphisms in $\mathcal{A}$ which is stable under pullbacks. Then $\mathcal{I}_G \cap \Proj(\mathcal{I}_G \cap\mathcal{M}) \subseteq \Proj(\mathcal{M})$. 
\end{lemma}
\begin{proof}
Suppose we are given a lifting problem
\[\xymatrix{
X \ar[d]_f \ar[r] & S \ar[d]^g \\ Y \ar[r] & T
}\]
with $f$ in $\mathcal{I}_G\cap \Proj(\mathcal{I}_G\cap \mathcal{M})$ and with $g\in \mathcal{M}$. Since $\mathcal{M}$ is stable under pullbacks, the map $Y\times_T S\rightarrow Y$ is in $\mathcal{M}$. By the universal property of the pullback, we have $X\leq Y\times_T S\leq Y$, so the fact that $GX \rightarrow GY$ is an isomorphism implies that $G(Y\times_T S) \rightarrow GY$ is an isomorphism. Hence the lifting problem
\[\xymatrix{
X \ar[d]_f \ar[r] & Y\times_T S \ar[d]^g \\ Y \ar[r] & T
}\]
is solvable, i.e., $Y\leq Y\times_T S$. Since $Y\times_T S \leq S$ as well, the original lifting problem is also solved.
\end{proof}

\begin{theorem}\label{main construction thm}
Given a compatible pair $(F,C)$ on a bicomplete preorder $\mathcal{A}$, there exists a model structure on $\mathcal{A}$ whose set of weak equivalences is $\mathcal{I}_{FC}$, and whose set of cofibrations is $\mathcal{Z}$. In that model structure, $F$ is the fibrant replacement monad, and $C$ is the cofibrant replacement comonad.
\end{theorem}
\begin{proof}
In the model structure claimed to exist, the acyclic cofibrations must be $\mathcal{Z}\cap \mathcal{I}_{FC}$, hence the fibrations must be $\Inj(\mathcal{Z}\cap\mathcal{I}_{FC})$, hence the acyclic fibrations must be equal to both $\Inj(\mathcal{Z})$ (so that they are precisely the maps with the right lifting property with respect to the cofibrations) and also $\mathcal{I}_{FC}\cap \Inj(\mathcal{Z}\cap\mathcal{I}_{FC})$ (so that they are the fibrations which are also acyclic). If we can prove the equality
\begin{align}\label{eq 309951}
 \mathcal{I}_{FC}\cap \Inj(\mathcal{Z}\cap\mathcal{I}_{FC})
  &= \Inj(\mathcal{Z}),
\end{align}
then the model structure exists: $\mathcal{I}_{FC}$ has the two-out-of-three property since it is decomposable, and the relevant factorization systems exist by the small object argument, and by the fact that $\mathcal{Z}$ and $\mathcal{Z}\cap\mathcal{I}_{FC}$ are cofibrantly closed classes in $\mathcal{A}$.

We verify the equality \eqref{eq 309951} as follows. One direction is easy: the containment $ \mathcal{I}_{FC}\cap \Inj(\mathcal{Z}\cap\mathcal{I}_{FC})\subseteq \Inj(\mathcal{Z})$ is transparently a consequence of (the dual of) Lemma \ref{lifting lemma}, since $\mathcal{Z}$ is stable under pushouts, since it is a cofibrantly closed class.

For the containment in the other direction: since $\mathcal{Z}\cap\mathcal{I}_{FC}\subseteq \mathcal{Z}$, we have $\Inj(\mathcal{Z}\cap\mathcal{I}_{FC})\supseteq \Inj(\mathcal{Z})$. We still need to know that $\Inj(\mathcal{Z})$ is also in $\mathcal{I}_{FC}$. For this, we argue as follows: given a morphism $f: X\rightarrow Y$ in $\Inj(\mathcal{Z})$, use the Stanculescu model structure to factor $f$ as a Stanculescu cofibration $X \rightarrow \tilde{Y}$ followed by a Stanculescu acyclic fibration $\tilde{Y} \rightarrow Y$. Since $\tilde{Y} \rightarrow Y$ is in $\mathcal{I}_{FC}$, we need only to show that $X \rightarrow \tilde{Y}$ is also in $\mathcal{I}_{FC}$. By Proposition \ref{fibrations are retractile}, the map $X \rightarrow \tilde{Y}$ is in $\Inj(\mathcal{Z})$. There are two possibilities: either $X \rightarrow \tilde{Y}$ is in $\mathcal{Z}$, or it isn't.
\begin{description}
\item[If $X\rightarrow\tilde{Y}$ is in $\mathcal{Z}$] then it is in both $\mathcal{Z}$ and in $\Inj(\mathcal{Z})$, hence it is an isomorphism, and consequence $X\rightarrow Y$ is in $\mathcal{I}_{FC}$, as desired.
\item[If $X\rightarrow\tilde{Y}$ is not in $\mathcal{Z}$] then, since $X\rightarrow \tilde{Y}$ is a Stanculescu cofibration, we have that $X\cong 0$ and that $\tilde{Y}$ is semicofibrant but not $C$-cofibrant. The map $0 \rightarrow C\tilde{Y}$, on the other hand, is in $\mathcal{Z}$, so the lifting problem
\[\xymatrix{
0 \ar[r]\ar[d] & X \ar[d] \\
C\tilde{Y} \ar[r] & \tilde{Y} 
}\]
is solvable, i.e., $C\tilde{Y}\leq X$. Hence $CX\leq C\tilde{Y} \cong CC\tilde{Y}\leq CX$, i.e., $X\rightarrow \tilde{Y}$ is in $\mathcal{I}_C\subseteq\mathcal{I}_{FC}$. Hence $X\rightarrow Y$ is the composite of two maps in $\mathcal{I}_{FC}$, hence it is in $\mathcal{I}_{FC}$, as desired.
\end{description}
We conclude that $\Inj(\mathcal{Z}\cap\mathcal{I}_{FC})\subseteq \Inj(\mathcal{Z})$, and hence the two sides of \eqref{eq 309951} are indeed equal.


Hence the model structure exists. We now must check that its cofibrant and fibrant objects are as claimed. From the definition of $\mathcal{Z}$, it is clear that the cofibrant objects in our model structure are precisely those $C$-cofibrant members of $\mathcal{A}$ which are also cofibrant in the Stanculescu model structure on $\mathcal{A}$. By Proposition \ref{stanculescu cofibrants}, the Stanculescu cofibrants are precisely the semicofibrants, and by (the dual of) Proposition \ref{semifibrancy characterization}, every $C$-cofibrant is a semicofibrant. Hence the cofibrants in our model structure are simply the $C$-cofibrants.

As for the fibrants in our model structure: compared to Stanculescu's model structure, our model structure has the same weak equivalences and no more cofibrations, so every Stanculescu fibrant must be fibrant in our model structure. By design, the Stanculescu fibrants are precisely the $F$-fibrants. Hence every $F$-fibrant is fibrant in our model structure.

We must show the converse as well. We argue by contrapositive: suppose that $Y\in\mathcal{A}$ is fibrant in our model structure, but not $F$-fibrant. Then some Stanculescu acyclic cofibration must fail to lift over $Y\rightarrow 1$. The Stanculescu acyclic cofibrations with non-initial domain are all acyclic cofibrations in our model structure as well, so the only possibility is that we have a non-solvable lifting problem
\begin{equation}\label{lifting problem 100301}\xymatrix{
0 \ar[r]\ar[d] & Y \ar[d] \\
X \ar[r] & 1
}\end{equation}
with $0 \rightarrow X$ a Stanculescu acyclic cofibration, i.e., a morphism in $\mathcal{I}_{FC}\cap\Proj(\mathcal{I}_{FC}\cap\Inj(\mathcal{I}_F))$, with $X$ {\em not} $C$-cofibrant, and with $Y\rightarrow 1$ a fibration in our model structure. Since $X$ is semicofibrant by Proposition \ref{stanculescu cofibrants}, we have $X\leq CFX$. Since $0\rightarrow X$ and $X\rightarrow CFX$ are each in $\mathcal{I}_{FC}$, the map $0\rightarrow CFX$ is an acyclic cofibration in our model structure, so the lifting problem
\[\xymatrix{
0 \ar[r]\ar[d] & Y \ar[d] \\
CFX \ar[r] & 1
}\]
{\em does} have a solution, i.e., $CFX\leq Y$. Now we have $X\leq CFX\leq Y$, i.e., the lifting problem \eqref{lifting problem 100301} has a solution after all. Hence the Stanculescu fibrants are all fibrant in our model structure. Hence the fibrants in our model structure are precisely the $F$-fibrant objects, as desired.
\end{proof}

\begin{theorem}\label{characterization thm}
Let $\mathcal{A}$ be a bicomplete preorder. Let $C$ be an idempotent comonad on $\mathcal{A}$, and let $F$ be an idempotent monad on $\mathcal{A}$. Then the following conditions are equivalent:
\begin{itemize}
\item There exists a model structure on $\mathcal{A}$ whose cofibrant replacement comonad is naturally isomorphic to $C$, and whose fibrant replacement monad is naturally isomorphic to $F$.
\item The pair $(F,C)$ is a compatible pair.
\end{itemize}
\end{theorem}
\begin{proof}
The implication in one direction is given by Proposition \ref{compatible pairs from model preorders}, and in the other direction, by Theorem \ref{main construction thm}.
\end{proof}

It is not generally the case that model structures on a bicomplete preorder are in one-to-one correspondence with compatible pairs. Theorem \ref{characterization thm} establishes that the former collection surjects on to the latter collection. This surjection is not necessarily injective. In the next subsection, we show that, if we restrict our attention to a certain natural class of model structures (the strong ones) and certain natural class of compatible pairs, we {\em do} get a one-to-one correspondence.

\subsection{Orthogonal pairs}
\label{Orthogonal pairs}

Our next task is to define orthogonal pairs. It will turn out (Corollary \ref{orth equals strongly}) that strong compatibility, defined in Definition \ref{def of compatible pair}, is in fact equivalent to orthogonality.


 \begin{definition}\label{def of orth cond}
     A compatible pair $(F,C)$ is said to satisfy \emph{the orthogonality condition} if, 
     for any lifting problem of the form
     \begin{equation}\label{orthogonal lifting problem}
\begin{tikzcd}
X \arrow[r] \arrow[d, "f"] & A \arrow[d, "g"] \\
Y \arrow[r]                              & B                             
\end{tikzcd}
\end{equation}
with $f\in \mathcal{I}_F$ and with $g\in \mathcal{I}_C$, a solution (equivalently, a unique solution) exists.
Compatible pairs satisfying the orthogonality condition will be called \emph{orthogonal pairs}.
 \end{definition}

Our next task is to prove Theorem \ref{CPs induce model structures}, which establishes a one-to-one correspodence between orthogonal pairs $(F,C)$ on a finitely bicomplete preorder $\mathcal{A}$ and strong model structures on $\mathcal{A}$. For that purpose, we will need some preliminary results. One such result is Proposition \ref{prelim chk thm}, which establishes that, given an idempotent monad $F$ on a finitely complete preorder $\mathcal{A}$, one can construct a factorization system on $\mathcal{A}$ whose left class is precisely the set $\mathcal{I}_F$ of morphisms inverted by the monad. This is roughly a special case of an old theorem \cite[Corollary 3.4]{MR779198} of Cassidy, H\'{e}bert, and Kelly, which applies to categories more generally, not only to preorders; however, the hypotheses of the Cassidy--H\'{e}bert--Kelly theorem include a certain set-theoretic assumption on the ambient category, namely, that the category must be ``finitely well-complete.'' This means that the category must be finitely complete, and it must admit all intersections of strong monomorphisms. 

For the purposes of this paper, we need only a version of the Cassidy--H\'{e}bert--Kelly theorem which applies to preorders, not general categories. In that level of generality, the assumption that the category is finitely well-complete is unnecessary: it suffices to assume that the preorder is finitely complete. Here is our claim:
\begin{prop}\label{prelim chk thm}
    Let $\mathcal{A}$ be a finitely complete preorder. Let $\mathcal{B}$ be a replete reflective subcategory of $\mathcal{A}$. Write $F: \mathcal{A}\rightarrow\mathcal{A}$ for the composite of the reflector functor $\mathcal{A}\rightarrow\mathcal{B}$ with the inclusion functor $\mathcal{B}\rightarrow\mathcal{A}$. Then there exists a factorization system $(\mathcal{L},\mathcal{R})$ on $\mathcal{A}$ such that $\mathcal{L}=\mathcal{I}_F$.
\end{prop}
We give a self-contained proof of Proposition \ref{prelim chk thm} to demonstrate the sufficiency of the hypothesis that the preorder is finitely complete. We claim no novelty, though: the same result, in greater generality, also appeared as Proposition 1.6 in \cite{MR1916156}.

\begin{proof} 
Let $\mathcal{L}$ be the class of morphisms in $\mathcal{A}$ which have the left lifting property with respect to every morphism in $\mathcal{B}$.
One verifies easily that the class $\mathcal{L}$ coincides with the class $\mathcal{I}_F$. 
Let $\mathcal{R} = \Inj(\mathcal{L})$. 

We must show that 
each morphism $f: X \rightarrow Y$ in $\mathcal{A}$ factors as a morphism in $\mathcal{L}$ followed by a morphism in $\mathcal{R}$.
To that end, consider the following commutative diagram
   \begin{center}
   \begin{equation}\label{chk diagram}
\begin{tikzcd}
A \arrow[rrd, "\eta_A", bend left] \arrow[rdd, "f"', bend right] \arrow[rd, "w"] &             &   \\
                & B \times  FA \arrow[d, "v"] \arrow[r, "u"'] &  FA \arrow[d, " Ff"] \\
                & B \arrow[r, "\eta_B"]                              &  FB                        
\end{tikzcd}
\end{equation}
   \end{center}
in which the evident square is a pullback square. We claim that the desired factorization of $f$ is the factorization $f = v\circ w$ along the left-hand side of \eqref{chk diagram}.
Since $F(f) \in  \mathcal{R}$, $v$ is a pullback of a morphism in $\mathcal{R}$, hence is also in $\mathcal{R}$. Since $A\leq B\times FA\leq FA$, we have $FA\cong F(B\times FA)$, i.e., $w$ is in $\mathcal{I}_F = \mathcal{L}$, as desired.
\end{proof}

We do not claim the next result is novel. It is the same as Corollary 3.4 in Cassidy--H\'{e}bert--Kelly \cite{MR779198}. except that we have different hypotheses. The category $\mathcal{A}$ is assumed in \cite[Corollary 3.4]{MR779198} to be a finitely well-complete category, i.e., a finitely complete category which admits all intersections of strong monomorphisms.  We are able to weaken assumption of finite well-completeness to finite completeness, because we include the additional hypothesis that $\mathcal{A}$ is a preorder. 


Here is the idea. 
Let $\mathcal{A}$ be a finitely complete preorder, let $RFS(\mathcal{A})$ be the set of isomorphism classes of retractile factorization systems on $\mathcal{A}$, and let $RRS(\mathcal{A})$ be the set of replete reflective subcategories of $\mathcal{A}$. It is well-known that the latter is in natural bijection with the collection of idempotent monads on $\mathcal{A}$ (even if $\mathcal{A}$ is a general category, not necessarily a preorder). We claim that $RRS(\mathcal{A)}$ and $RFS(\mathcal{A})$ are also in bijection with each other.

An explicit bijection is as follows. Given a factorization system $(\mathcal{L},\mathcal{R})$ on $\mathcal{A}$, call an element $X\in\mathcal{A}$ {\em $\mathcal{R}$-fibrant} if the map $X\rightarrow 1$ is in $\mathcal{R}$.
Now let $\xi: RFS(\mathcal{A}) \rightarrow RRS(\mathcal{A})$ be the function which sends a retractile factorization system $(\mathcal{L},\mathcal{R})$ to the subset of $\mathcal{A}$ consisting of the $\mathcal{R}$-fibrant elements.
In the other direction, let $\zeta: RRS(\mathcal{A}) \rightarrow RFS(\mathcal{A})$ be the function which sends a reflective subcategory $\mathcal{B}$ of $\mathcal{A}$ to the factorization system $(\mathcal{L},\mathcal{R})$ constructed in Proposition \ref{prelim chk thm}, i.e., $\mathcal{L}$ is the set of morphisms in $\mathcal{A}$ inverted by the reflector functor $\mathcal{A}\rightarrow\mathcal{B}$. This characterization of $\mathcal{L}$ also makes it clear that the class of morphisms $\mathcal{L}$ is retractile, i.e., the factorization system $(\mathcal{L},\mathcal{R})$ is retractile, as claimed.

It is routine to check that the functions $\zeta$ and $\xi$ are mutually inverse, yielding the result:
\begin{theorem}\label{idem monads are bijective to retractile fact sys}
    For any finitely bicomplete preorder $\mathcal{A}$, the functions $\zeta: RRS(\mathcal{A}) \rightarrow RFS(\mathcal{A})$ and $\xi: RFS(\mathcal{A}) \rightarrow RRS(\mathcal{A})$ are mutually inverse. Consequently, idempotent monads on $\mathcal{A}$ are in one-to-one correspondence with retractile factorization systems on $\mathcal{A}$.
\end{theorem}

Recall that, in Definition \ref{def of strong model structures}, we defined a {\em strong} model structure to be one whose acyclic fibrations are sectile and whose acyclic cofibrations are retractile. Equivalently: a model structure is strong if and only if its acyclic fibrations and acyclic cofibrations are decomposable (Definition \ref{definition of a center}). 
\begin{theorem}\label{CPs induce model structures}
Let $\mathcal{A}$ be a finitely bicomplete preorder, and let $(F,C)$ be an orthogonal pair on $\mathcal{A}$. Then there exists a strong model structure on $\mathcal{A}$ whose acyclic fibrations are precisely the morphisms inverted by $F$, and whose acyclic cofibrations are precisely the morphisms inverted by $C$. The fibrant replacement monad and cofibrant replacement comonad of this model structure are $F$ and $C$, respectively.

Furthermore, every strong model structure on $\mathcal{A}$ arises from this construction. Consequently we have a bijection between strong model structures on $\mathcal{A}$, and isomorphism classes of orthogonal pairs on $\mathcal{A}$.
\end{theorem}
\begin{proof}
Apply Theorem \ref{prelim chk thm} to the idempotent monad $F$ to get a factorization system $(\mathcal{L}_F, \mathcal{R}_F)$ on $\mathcal{A}$. Apply Theorem \ref{prelim chk thm} once more (in its dual form), to the idempotent comonad $C$, to get a factorization system $(\mathcal{L}_C, \mathcal{R}_C)$ on $\mathcal{A}$. 

We define a model structure on $\mathcal{A}$ as follows:
\begin{itemize}
\item The acyclic cofibrations will be $\mathcal{L}_F$, i.e., $\mathcal{I}_F$.
\item The fibrations consequently must be $\mathcal{R}_F$, i.e., $\Inj(\mathcal{I}_F)$.
\item The acyclic fibrations will be $\mathcal{R}_C$, i.e., $\mathcal{I}_C$.
\item The cofibrations consequently must be $\mathcal{L}_C$, i.e., $\Proj(\mathcal{I}_C)$.
\item The weak equivalences consequently must be precisely those morphisms which factor as the composite of a morphism in $\mathcal{I}_F$ followed by a morphism in $\mathcal{I}_C$. We will write $W$ for this set of morphisms.
\end{itemize}
Since $\mathcal{I}_F$ and $\mathcal{I}_C$ are each decomposable, this is clearly a strong model structure if it is a model structure at all.


To verify that we have indeed defined a model structure on $\mathcal{A}$, we make the following sequence of observations:
\begin{enumerate}
\item The containment $\mathcal{L}_F\subseteq\mathcal{L}_C$ is equivalent to the orthogonality condition (Definition \ref{def of orth cond}) on the compatible pair $(F,C)$. Hence we have $\mathcal{L}_F\subseteq\mathcal{L}_C$, i.e., every acyclic cofibration is indeed a cofibration. This implies that $\mathcal{R}_C\subseteq\mathcal{R}_F$, i.e., every acyclic fibration is a fibration as well.
\item  We claim that $W = \mathcal{I}_{FC}$, i.e., the weak equivalences are precisely the morphisms inverted by the composite functor $FC$. The inclusion $W\subseteq \mathcal{I}_{FC}$ is easy to see. For the reverse containment: given a morphism $X \rightarrow Y$ in $\mathcal{I}_{FC}$, factor that morphism into a map $X\rightarrow X^{\prime}$ in $\mathcal{L}_F$ followed by a map $X^{\prime}\rightarrow Y$ in $\mathcal{R}_F$. 
Since $X\rightarrow Y$ is inverted by $CF$, and since $\mathcal{I}_{CF}$ is decomposable, the map $X^{\prime} \rightarrow Y$ is also inverted by $CF$. Hence $CX^{\prime} \rightarrow CY$ is in $\mathcal{L}_F$, hence a lift exists in the square
\[\xymatrix{
 CX^{\prime} \ar[d]_{\mathcal{L}_F} \ar[r] & X^{\prime} \ar[d]^{\mathcal{R}_F} \\
 CY \ar[r]\ar@{-->}[ur] & Y. }\]
Since $\mathcal{R}_C = \mathcal{I}_C$ is decomposable, and since $CY\rightarrow Y$ is in $\mathcal{R}_C$ and factors as $CY \rightarrow X^{\prime}\rightarrow Y$, the map $X^{\prime}\rightarrow Y$ is in $\mathcal{R}_C$ as well. 

Hence we have factored $X \rightarrow Y$ into a map $X\rightarrow X^{\prime}$ in $\mathcal{L}_F$ followed by a map $X^{\prime}\rightarrow Y$ in $\mathcal{R}_C$, i.e., $\mathcal{I}_{FC} \subseteq W$, as desired.

Since $W = \mathcal{I}_{FC}$ is decomposable, the weak equivalences satisfy two-out-of-three.
\item We must check that the acyclic cofibrations are precisely the weak equivalences which are cofibrations, i.e., $\mathcal{L}_F = \mathcal{I}_{FC}\cap \mathcal{L}_C$. 

Consider any map $X \xrightarrow{f}Y \in \mathcal{L}_C \, \cap \, \mathcal{W}$. We can factor $f$ as $X \xrightarrow{g} Z \xrightarrow{h} Y$, where $g \in \mathcal{L}_F$ and $h \in \mathcal{R}_C$. 
By the usual retract argument \cite[Lemma 1.1.9]{MR1650134}, $h$ is an isomorphism, so $f \in \mathcal{L}_F$, as desired.

Conversely, if $X \xrightarrow{f}Y \in \mathcal{L}_F$, then $Ff$ is an isomorphism, hence $CFf$ is an isomorphism, and hence $f$ is a weak equivalence. To verify that $f$ is also in $\mathcal{L}_C$, we must check that, for every map $A\rightarrow B$ in $\mathcal{R}_C$ (i.e., every map $A\rightarrow B$ inverted by $C$), the lifting problem
\begin{center}
\begin{tikzcd}
X \arrow[r] \arrow[d] & A \arrow[d] \\
Y \arrow[r]\arrow[ur,dashed]           & B                             
\end{tikzcd}
\end{center}
has a solution. This solution is provided precisely by the assumption that $(F,C)$ is an orthogonal pair.
\item The equality $\mathcal{R}_F = \mathcal{I}_{FC}\cap \mathcal{R}_C$ is proven by an argument dual to the argument just given.
\end{enumerate}

Hence, we indeed have a strong model structure associated to each orthogonal pair. The reverse direction is easier: given a strong model structure on $\mathcal{A}$, write $F$ for the fibrant replacement monad and $C$ for the cofibrant replacement comonad of the model structure. It is an easy exercise, using the assumption of strongness of the model structure, to verify that the set of acyclic cofibrations must be precisely $\mathcal{I}_F$, and the set of acyclic fibrations must be precisely $\mathcal{C}$. Hence every lifting problem of the form \eqref{orthogonal lifting problem} has a solution, i.e., $(F,C)$ is an orthogonal pair.
\end{proof}

\begin{corollary}\label{orth equals strongly}
    A compatible pair on a finitely bicomplete preorder is orthogonal if and only if it is strongly compatible.
\end{corollary}
\begin{proof}
Consider an orthogonal pair $(F,C)$. Then if $X \xrightarrow{f}Y$ is a map inverted by both $F$ and $C$, the following lifting diagram has a solution:
    \begin{center}
        \begin{tikzcd}
X \arrow[r, "id_X"] \arrow[d, "f"'] & X \arrow[d, "f"] \\
Y \arrow[r, "id_Y"']                & Y               
\end{tikzcd}.
    \end{center}
    Consequently $f$ is an isomorphism. Hence, every orthogonal pairs is strongly compatible.

    The converse is also true since by Theorem \ref{CPs induce model structures}, the existence of a model structure induced by the strongly compatible pair $(F,C)$ guarantees that all the lifting problems of the form of \eqref{orthogonal lifting problem} have a solution, i.e., the pair $(F,C)$ is orthogonal.
\end{proof}

\section{Strong model structures on finite Boolean algebras} 
\label{Strong model structures on finite Boolean algebras}

\subsection{All strong model structures arise from pairs of Moore collections}
\label{All strong model structures...}

In this subsection we review old and well-known connections between closure algebras, Moore collections, idempotent monads, and (somewhat less well-known, but nevertheless derivable from \cite{MR779198}) factorization systems on power sets. We learned these ideas from chapter 10 of the remarkable book \cite{MR1858927}. A more modern review, which also provides some generalizations, is also given in \cite{BCH+25}.

The Boolean algebras are a famous class of preorders: order-theoretically, the Boolean algebras are characterized by being precisely the complemented distributive lattices. The category of Boolean algebras is equivalent to the category of Boolean rings \cite{stone1935subsumption}, i.e., the commutative rings in which every element is idempotent. It is easy to use this algebraic characterization to show that every {\em finite} Boolean algebra is isomorphic, as a partially-ordered set, to the power set $\mathcal{P}(X)$ of some finite set $X$. 

Hence, by Theorem \ref{CPs induce model structures}, strong model structures on finite Boolean algebras are in bijection with orthogonal pairs on the power sets of a finite sets. In any orthogonal pair $(F,C)$ on the power set $\mathcal{P}(X)$, $F$ will be an idempotent monad on $\mathcal{P}(X)$, and $C$ will be an idempotent comonad on $\mathcal{P}(X)$. In order-theoretic terms, an idempotent monad on $\mathcal{P}(X)$ is an idempotent, increasing (i.e., $Y\subseteq Z$ implies $FY\subseteq FZ$), extensive (i.e., $Y\subseteq FY$) function from $\mathcal{P}(X)$ to $\mathcal{P}(X)$. In order theory, such a function is called a {\em closure operator}. 

It is an old and elementary observation\footnote{It seems to go back to E. H. Moore's 1910 book \cite{moore1910introduction}, although we do not know whether that is the origin of the idea. A more easily-obtained reference is chapter V of Birkhoff's influential book on lattice theory \cite{MR227053}.}
that one can specify a closure operator $F$ on $X$ by specifying which subsets $Y$ of $X$ are ``closed'' in the sense that $Y = FY$. The resulting collection of closed subsets of $X$ is not necessarily the collection of closed sets of a topology: it might not be true that the empty set is closed, and it might not be true that a union of finitely many closed sets is again closed. It is, however, at least true that the closed sets form a {\em Moore collection}:
    \begin{definition}
        A collection of subsets of a set $X$ is a \textit{Moore collection} if it is closed under arbitrary intersections. 

        Dually, a collection of subsets of $X$ is a {\em co-Moore collection} if it is closed under arbitrary unions.
    \end{definition}
From a Moore collection $\mathcal{U}\subseteq\mathcal{P}(X)$, one can recover the idempotent monad $F$ on $\mathcal{P}(X)$ by letting $FY$ be the intersection of all the closed sets (i.e., members of $\mathcal{U}$) containing $Y$. Hence, Moore collections in $X$ are in bijection with idempotent monads on $\mathcal{P}(X)$, consequently also in bijection with reflective subcategories of $\mathcal{P}(X)$, and consequently (by Theorem \ref{idem monads are bijective to retractile fact sys}) also in bijection with retractile factorization systems on $\mathcal{P}(X)$. Dually, co-Moore collections in $X$ are in bijection with idempotent comonads, coreflective subcategories, and sectile factorization systems on $\mathcal{P}(X)$.

\subsection{Model structures arising from topologies}
\label{Model structures arising from topologies}

Suppose that $\mathcal{S}$ is a reflective subcategory of $\mathcal{P}(X)$, with reflector functor $F$. Then $\mathcal{S}$ inherits a partial ordering from $\mathcal{P}(X)$, and in this partial ordering, $\mathcal{S}$ is a lattice. However, the join of two elements is $F(A\vee B)$, where $A\vee B$ is the join in $\mathcal{P}(X)$ (i.e., $A \vee B = A\cup B$). The point is that the join in $\mathcal{S}$ might not agree with the join in $\mathcal{P}(X)$. Hence, while $\mathcal{P}(X)$ is a lattice, and $\mathcal{S}$ is a lattice, and $\mathcal{S}$ is a sub-poset of $\mathcal{P}(X)$, it might not be true that $\mathcal{S}$ is a sub-lattice of $\mathcal{P}(X)$!

On the other hand, sometimes $\mathcal{S}$ {\em is} a sub-lattice of $\mathcal{P}(X)$. This happens if and only if the corresponding Moore collection is closed under finite unions. We need only add the condition that $\mathcal{S}$ contains the bottom element $\emptyset \in\mathcal{P}(X)$, and now the Moore collection in question is precisely a topology on $X$. This observation motivates the following definition:
\begin{definition}
Let $X$ be a set. A reflective subcategory $\mathcal{S}$ of $\mathcal{P}(X)$ will be called a {\em topological subcategory of $\mathcal{P}(X)$} if the Moore collection associated to $X$ is a topology on $X$. 
Equivalently:
\begin{itemize}
\item The reflective subcategory $\mathcal{S}$ is topological if and only if its reflector functor $F$ commutes with finite colimits.
\item The reflective subcategory $\mathcal{S}$ is topological if and only if $F(\emptyset) = \emptyset$, and $F(Y\cup Z) = (FY) \cup (FZ)$ for all $Y,Z\in\mathcal{P}(X)$.
\end{itemize}
\end{definition}

A set $X$ equipped with a closure operator on its power set $\mathcal{P}(X)$ is standardly called a {\em closure algebra}. The dual notion, a set equipped with an interior operator (i.e., an idempotent comonad) on its power set, is called an {\em interior algebra}. Evidently a model structure on the Boolean algebra $\mathcal{P}(X)$ equips $X$ with the structure of a closure algebra (via the fibrant replacement monad) and an interior algebra (via the cofibrant replacement comonad). In light of Theorem \ref{CPs induce model structures}, there is then a one-to-one correspondence between strong model structures on $\mathcal{P}(X)$ and pairs $(\Cl,\Int)$, where $\Cl$ is a closure operator on $X$, and $\Int$ is an interior algebra on $X$, satisying compatibility conditions equivalent to compatibility (Definition \ref{def of compatible pair}) and orthogonality (Definition \ref{def of orth cond}) of the monad-comonad pair $(\Cl,\Int)$.
It is an easy matter to write down those conditions:
\begin{prop}\label{closure-interior compatibility}
Let $X$ be a set. Strong model structures on $\mathcal{P}(X)$ are in bijection with ordered pairs $(\Cl,\Int)$, where $\Cl$ is a closure operator on $X$, and $\Int$ is an interior operator on $X$, satisfying the conditions:
\begin{itemize}
\item 
For each subset $U$ of $X$, we have $\Cl(\Int(U)) = \Int(\Cl(U))$.
\item 
If $U,V$ are subsets of $X$ such that $U\subseteq \Int(V)$ and $\Cl(U)\subseteq V$, then $\Cl(U)\subseteq \Int(V)$.
\item 
If $U,V$ are subsets of $X$ such that $U\subseteq V$ and $\Cl(U) = \Cl(V)$ and $\Int(U) = \Int(V)$, then $U=V$.
\end{itemize}
\end{prop}

In particular, some of the strong model structures on $\mathcal{P}(X)$ arise by means of a pair of topologies $\mathcal{T}_1,\mathcal{T}_2$ on $X$, such that, if we write $\Cl$ for the closure operator in the topology $\mathcal{T}_1$ (presented via its closed sets), and $\Int$ for the interior operator in the topology $\mathcal{T}_2$, then the conditions of Proposition \ref{closure-interior compatibility} are satisfied. We adopt a name for these model structures:
\begin{definition}\label{topological def}
A factorization system on $\mathcal{P}(X)$ will be called {\em topological} if it arises in this way from a topology on $X$.
Similarly, a strong model structure on $\mathcal{P}(X)$ will be called {\em topological} if it arises in this way from a pair of topologies on $X$.
\end{definition}
We think that context is likely to be sufficient to avoid any confusion between topological model structures on Boolean algebras, as in Definition \ref{topological def}, and topological model categories in the sense of a model category enriched in topological spaces and satisfying appropriate axioms as in \cite[Definition VII.4.2]{MR1417719}. 

Write $\Moore(X)$ for the set of Moore collections on a set $X$, and write $\Top(X)$ for the set of topologies on $X$. The set of strong model structures on $\mathcal{P}(X)$ is a subset of $\Moore(X)\times \Moore(X)$, by Proposition \ref{closure-interior compatibility}. It is natural to ask how much of $\Moore(X)\times \Moore(X)$ is comprised by this subset; roughly, how likely is it that a randomly-chosen pair of Moore collections in $X$ defines a model structure on $\mathcal{P}(X)$? Similarly, the set of topological model structures on $\mathcal{P}(X)$ is a subset of $\Top(X)\times\Top(X)$, and one naturally asks how large this subset is.

Here we attempt to answer these questions only for very small sets $X$, and only by a na\"{i}ve approach, via direct calculation. One can use a computer to do the following:
\begin{enumerate}
\item Generate all Moore families on $X$ using the depth-first algorithm as described in \cite{MR1489091}.
\item For each Moore family $\mathcal{U}$ on $X$, check whether $\mathcal{U}$ is a topology.
\item For each ordered pair $(\mathcal{U},\mathcal{V})$ of Moore families on $X$, consider the ordered pair $(\mathcal{U},\mathcal{V}^{\op})$, where $\mathcal{V}^{\op}$ is the co-Moore family in $X$ dual to $\mathcal{V}$. Check whether $(\mathcal{U},\mathcal{V}^{\op})$ satisfies the conditions of Proposition \ref{closure-interior compatibility}, so that it defines a strong model structure on $\mathcal{P}(X)$.
\item Same as the previous step, but with $\mathcal{U},\mathcal{V}$ each taken from the list of topologies on $X$.
\end{enumerate}
Our implementation of the above procedure in Sage \cite{sagemath95}, available at \url{https://github.com/Gunjeet-Singh/Generating-compatible-monad-comonad-pairs}, yielded the following counts:
\begin{equation}\label{tab:strong_model_structures}
\begin{tabular}{|c|c|c|c|c|}
\hline
\text{Cardinality} & \text{Number of} & \text{Number of} & \text{Number of} & \text{Number of} \\
\text{of $X$} & \text{Moore} & \text{topologies} & \text{strong model} & \text{topological} \\
 & \text{families in $X$} & \text{on $X$} & \text{structures} & \text{model structures} \\
 & & &  \text{on $\mathcal{P}(X)$} & \text{on $\mathcal{P}(X)$}\\
\hline
1 & 2     & 1     & 3      & 1    \\
2 & 7     & 4     & 17     & 9   \\
3 & 61    & 29    & 377    & 84  \\
4 & 2480  & 355   & 127866 & 1295 \\
\hline
\end{tabular}
\end{equation}
The second and third columns, counting Moore collections and topologies on a finite set, are well-studied classically; e.g. \cite[A102896,A000798]{oeis2025}. The other columns seem to not have been calculated before, and there is no obvious description of their values in terms of familiar sequences from combinatorics.
     
The ratios of the number of strong model structures to the size of $\Moore(X)\times \Moore(X)$, and of the number of topological model structures to the size of $\Top(X)\times\Top(X)$, are as follows:
\begin{equation}
\centering
\begin{tabular}{|c|c|c|}
\hline
\text{Cardinality of $X$} & 
\text{\shortstack{Strong model\\ structures as \% \\ of $(\text{Moore})^2$}} & 
\text{\shortstack{Topological model\\ structures as \\ \% of strong\\ model structures}} \\
\hline
1 & 33.33\% & 33.33\% \\
2 & 34.69\% & 52.94\% \\
3 & 10.14\% & 22.29\% \\
4 & 2.08\%  & 1.01\% \\
\hline
\end{tabular}
\label{tab:strong_model_structure_percentages}
\end{equation}
It is natural to guess that both ratios go to zero as the size of the finite set $X$ goes to $\infty$, but we have not attempted to find a proof.

\subsection{Model structures arising from geometries and matroids}

Moore collections in $X$ are in bijection with retractile factorization systems in $\mathcal{P}(X)$, as explained in \cref{All strong model structures...}. In mathematics, the most commonly-considered kind of Moore collection is a topology. However, topologies are not the only kind of Moore collection which has attracted interest and study: in particular, of the several equivalent definitions of a matroid, one definition of a matroid is as a particular kind of Moore collection. The matroids which are ``simple,'' in a sense to be defined below (Definition \ref{def of matroid}), are called ``geometries,'' and are studied in combinatorial approaches to projective geometry. 

In this subsection, we recall the definitions of matroids and geometries in terms of Moore collections, and we explain how appropriate pairs of matroids and geometries give rise to model structures on finite Boolean algebras.

\begin{definition}\label{def of matroid}
A {\em matroid} is a Moore collection on a set $X$ satisfying the following two conditions. In these conditions, we write $\Cl(A)$ for the closure of a subset $A$ of $X$ in the given Moore collection.
\begin{description}
\item[Finitary property] If $x\in \Cl(A)$ for some element $x\in X$ and some subset $A\subseteq X$, then $x\in \Cl(B)$ for some finite subset $B$ of $X$. 
\item[Exchange property] If $x,y\in X$ and $A\subseteq X$ and $x\notin \Cl(A)$ and $x\in \Cl(A\cup \{ y\})$, then $y\in \Cl(A\cup \{x\})$.
\end{description}
The closed sets in a matroid are sometimes called {\em subspaces} or {\em flats.}

A matroid is called a {\em simple matroid} or a {\em geometry} if it also satisfies the simpleness condition:
\begin{description}
\item[Simpleness property] $\Cl(\emptyset) = \emptyset$, and for each $x\in X$, we have $\Cl(\{x\}) = \{x\}$.
\end{description}
\end{definition} 
If the set $X$ has at least two elements, then the condition $\Cl(\emptyset) = \emptyset$ is implied by the rest, and can be omitted. If $X$ is also finite, then of course the finitary property is automatically satisfied, and can be omitted.

The book \cite{MR1783451} of Faure and Fr\"{o}licher is a good textbook reference for these ideas. A treatment of matroids in terms of Moore collections, in particular, is given in chapter 3 of \cite{MR1783451}. For a shorter introduction to the significance and applications of matroids, one can consult E. Katz's paper \cite{MR3702317} for connections to algebraic geometry, and the material in the introduction preceding the statements of any theorems in Biss's paper \cite{MR2031856} for connections to algebraic topology. (There are well-known problems with the new results in \cite{MR2031856}, as indicated by \cite{mnev2007dkbisspapersthe} and \cite{MR2521124}. Nevertheless, the expository material at the start of \cite{MR2031856} is informative, compelling, and quite well-written.)

\begin{prop}\label{matroidal fact systems}
Let $X$ be a set. Then the collection of matroids with underlying set $X$ embeds into the collection of retractile factorization systems on the power set $\mathcal{P}(X)$. The factorization system $(\mathcal{L},\mathcal{R})$ associated to a matroid $\mathcal{M}$ is given by letting $\mathcal{L}$ be the set of all maps $U\subseteq V$ in $\mathcal{P}(V)$ such that $\Cl(U) = \Cl(V)$. 
\end{prop}

As a corollary of Proposition \ref{closure-interior compatibility}, we have:
\begin{prop}\label{model structs from matroids}
Let $X$ be a set. Suppose $\mathcal{M}_1,\mathcal{M}_2$ are matroids with underlying set $X$. Write $\Cl$ for the closure operator in $\mathcal{M}_1$, and write $\Int$ for the interior operator in $\mathcal{M}_2$, i.e., $\Int(U) = c(\Cl_2(cU))$, where $c$ is complementation, and $\Cl_2$ is the closure operator of $\mathcal{M}_2$. Then the following conditions are equivalent:
\begin{itemize}
\item There exists a strong model structure on the power set $\mathcal{P}(X)$ whose fibrant objects are precisely the closed sets in the matroid $\mathcal{M}_1$, and whose cofibrant objects are precisely the complements of closed sets in the matroid $\mathcal{M}_2$.
\item There exists a {\em unique} strong model structure on the power set $\mathcal{P}(X)$ whose fibrant objects are precisely the closed sets in the matroid $\mathcal{M}_1$, and whose cofibrant objects are precisely the complements of closed sets in the matroid $\mathcal{M}_2$.
\item The following conditions each hold:
\begin{align*}
 \Cl( \Int(U)) &= \Int(\Cl(U))\mbox{\ for\ all\ }U\subseteq X,\\
\mbox{if\ } U\subseteq \Int(V)\mbox{\ and\ }& \Cl(U) \subseteq V,\mbox{then\ } \Cl(U)\subseteq \Int(V),\mbox{\ and} \\
\mbox{if\ } U\subseteq V\mbox{\ and\ } \Cl(U)=\Cl(V)\mbox{\ and\ }& \Int(U) = \Int(V),\mbox{then\ } U=V.
\end{align*}
\end{itemize}
\end{prop}

\begin{definition}\label{def of matroidal and geometric}
A factorization system on a Boolean algebra will be called {\em matroidal} (respectively, {\em geometric}) if it can be constructed from a matroid (respectively, geometry) $\mathcal{M}$ by the construction in Proposition \ref{matroidal fact systems}. Similarly, a model structure on a Boolean algebra will be called {\em matroidal} (respectively, {\em geometric}) if it can be constructed from an ordered pair of matroids (respectively, geometries) $(\mathcal{M}_1,\mathcal{M}_2)$ by the construction in Proposition \ref{model structs from matroids}.
\end{definition}

\section{Model structures on preorders, in general.}
\label{factorization bracket pairs section}

By Theorem \ref{CPs induce model structures}, the data of a strong model structure on a finitely bicomplete preorder $\mathcal{A}$ is entirely encoded by a choice of an idempotent monad and an idempotent comonad on $\mathcal{A}$ satisfying the compatibility conditions (Definition \ref{def of compatible pair}) and the orthogonality condition (Definition \ref{def of orth cond}). It is natural to ask whether some similarly ``algebraic'' structure on $\mathcal{A}$ might describe {\em all} model structures on $\mathcal{A}$ in general, not only the strong model structures. Indeed, such an algebraic description is possible: we call it a {\em factorization bracket pair}. It is not very deep at all---in fact, this algebraic structure is quite transparently equivalent to the data of a model structure---but nevertheless specifying model structures via this algebraic structure is useful for the computer calculations we describe in \cref{Explicit examples}.
\begin{definition}\label{def of fact bracket}
 Let $\mathcal{A}$ be a preorder, and let $\Rels(\mathcal{A})$ denote the set of ordered pairs $(X,Y)$ of objects of $\mathcal{A}$ such that $X\leq Y$.
\begin{itemize}
 \item By a {\em factorization bracket on $\mathcal{A}$} we mean a function $\Phi:\Rels(\mathcal{A})\rightarrow \mathcal{A}$ satisfying the four axioms:
\begin{align}
\label{f binary axiom 1}  X\leq & \Phi(X,Y) \leq Y, \\
\label{f trinary axiom 1}  \mbox{if\ }X\leq Y \mbox{\ and\ } \Phi(X,Y)\leq Z, &
   \mbox{\ then\ } \Phi(\Phi(X,Y),Z) \cong \Phi(X,Z),\\
\label{f trinary axiom 2}  \mbox{if\ }Y\leq Z \mbox{\ and\ } X\leq \Phi(Y,Z), &
   \mbox{\ then\ } \Phi(X,\Phi(Y,Z)) \cong \Phi(X,Z), \\
\label{f quaternary axiom 1}  \mbox{if\ }W\leq \Phi(Y,Z)\mbox{\ and\ } \Phi(W,X)\leq Z,&
   \mbox{\ then\ } \Phi(W,X)\leq \Phi(Y,Z).
 \end{align}
 \item By a {\em factorization bracket pair on $\mathcal{A}$} we mean an ordered pair $\Phi,\Psi$ of factorization brackets on $\mathcal{A}$ satisfying the five axioms
 \begin{align}
\label{fbpair binary axiom 1}   \mbox{if\ } X\leq Y,\mbox{\ then\ } \Psi\left(X,\Phi(X,Y)\right)&\cong \Phi(X,Y),\\
\label{fbpair binary axiom 2}   \mbox{if\ } X\leq Y,\mbox{\ then\ } \Phi\left(\Psi(X,Y),Y\right) &\cong \Psi(X,Y),\\ 
\label{fbpair trinary axiom 4}   \mbox{if\ } X\leq Y\leq Z, \mbox{\ then\ } \Phi\left(\Psi(X,Y),\Phi(Y,Z)\right) &\cong \Psi\left(\Psi(X,Y),\Phi(Y,Z)\right),\\
\label{fbpair trinary axiom 5}  \mbox{if\ } X\leq Y\leq Z\mbox{\ and\ }\hspace{80pt} &\\ \nonumber \hspace{20pt} \Phi\left( \Psi(X,Y),\Psi(Y,Z)\right) \cong \Psi\left( \Psi(X,Y),\Psi(Y,Z)\right), &\mbox{\ then\ } \Phi\left(Y,\Psi(Y,Z)\right)\cong \Psi(Y,Z),\\
\label{fbpair trinary axiom 6}   \mbox{if\ } X\leq Y\leq Z\mbox{\ and\ }\hspace{80pt} &\\ \nonumber \hspace{20pt}  \Phi\left( \Phi(X,Y),\Phi(Y,Z)\right) \cong \Psi\left( \Phi(X,Y),\Phi(Y,Z)\right), &\mbox{\ then\ } \Psi\left(\Phi(X,Y),Y\right)\cong \Phi(X,Y).
 \end{align}
 \end{itemize}
 \end{definition}

 \begin{definition-proposition}
 Given a factorization bracket pair $\Phi,\Psi$ on a preorder $\mathcal{A}$ with initial object $0$ and terminal object $1$, its {\em associated compatible pair} consists of the monad $F$ on $\mathcal{A}$ given by $F(X) = \Phi(X,1)$ and the comonad $C$ on $\mathcal{A}$ given by $C(X) = \Psi(0,X)$. 
 \end{definition-proposition}
 \begin{proof}
 We should check that $F,C$ is indeed a compatible pair. We check the axioms from Definition \ref{def of compatible pair} in order:
 \begin{itemize}
 \item $FCX = \Phi(\Psi(0,X),1) \cong \Psi(0,\Phi(X,1)) = CFX$, so $F$ commutes with $C$ up to isomorphism.
 \item If $X,Y$ are objects of $\mathcal{A}$ such that $X\leq CY$ and $FX\leq Y$, then $X\leq CY = \Psi(0,Y)$ and $\Phi(X,1) = FX\leq Y$, so $FX = \Phi(X,1)\leq \Psi(0,Y) = CY$, as desired.
 \end{itemize}
 \end{proof}

\begin{prop}\label{fact brackets and fact systems and model structs}
Isomorphism classes of factorization brackets on a preorder $\mathcal{A}$ are in bijection with factorization systems on $\mathcal{A}$.
Similarly, isomorphism classes of factorization bracket pairs on a preorder $\mathcal{A}$ are in bijection with model structures on $\mathcal{A}$.
\end{prop}
\begin{proof}
Given a factorization system $(\mathcal{L},\mathcal{R})$, we get a factorization bracket $\Phi$ on $\mathcal{A}$ by letting $\Phi(X,Y)$ be ``the'' (unique up to isomorphism) element of $\mathcal{A}$ which fits into a factorization $X\leq \Phi(X,Y)\leq Y$ with $X\leq \Phi(X,Y)$ in $\mathcal{L}$ and with $\Phi(X,Y)\leq Y$ in $\mathcal{R}$. 
Going the other way, i.e., beginning with a factorization bracket $\Phi$, we claim that one gets a factorization system $(\mathcal{L},\mathcal{R})$ by letting $\mathcal{L}$ be the set of all morphisms $X\leq Y$ such that $\Phi(X,Y)\cong Y$, and letting $\mathcal{R}$ be the set of all morphisms $X\leq Y$ such that $X\cong \Phi(X,Y)$. 

Checking the axioms in Definition \ref{def of fact bracket} is routine.
We conclude that factorization brackets are equivalent to factorization systems.

A model structure on $\mathcal{A}$ is given by a pair of factorization systems $(cof,fib\ we)$ and $(cof\ we, fib)$ satisfying a few axioms:
\begin{itemize}
\item $cof\ we\subseteq cof$,
\item $fib\ we\subseteq fib$,
\item and the weak equivalences---which are understood to be the maps $X\leq Z$ such that $\Phi(X,Z) = \Psi(X,Z)$, i.e., the maps $X\leq Z$ which factor as a composite $X\leq Y\leq Z$ such that $X\leq Y$ is in $cof\ we$ and $Y\leq Z$ is in $fib\ we$---satisfy the two-out-of-three axiom.
\end{itemize}
Now we compare those axioms with the axioms in the definition of a factorization bracket pair, from Definition \ref{def of fact bracket}:
\begin{itemize}
\item Axiom \eqref{fbpair binary axiom 1} is equivalent to the containment $cof\ we\subseteq cof$, while axiom \eqref{fbpair binary axiom 2} is equivalent to $fib\ we\subseteq fib$.
\item Axiom \eqref{fbpair trinary axiom 4} is equivalent to the statement that, when you factor a composite of the form $\bullet \stackrel{fib\ we}{\longrightarrow} \bullet \stackrel{cof\ we}{\longrightarrow} \bullet$
into a composite of the form $\bullet \stackrel{cof}{\longrightarrow} \bullet \stackrel{fib\ we}{\longrightarrow} \bullet$ (respectively, a composite of the form $\bullet \stackrel{cof\ we}{\longrightarrow} \bullet \stackrel{fib}{\longrightarrow} \bullet$), the remaining $cof$ map is also in $cof\ we$ (respectively, the remaining $fib$ map is also in $fib\ we$). This is, in turn, equivalent to the weak equivalences being closed under composition, which is part of the two-out-of-three axiom.
\item Axiom \eqref{fbpair trinary axiom 5} is equivalent to the condition that, given a composite of the form $\bullet \stackrel{fib\ we}{\longrightarrow} \bullet \stackrel{cof}{\longrightarrow} \bullet$ which is a weak equivalence, the map marked $cof$ is also in $cof\ we$. Similarly, axiom \eqref{fbpair trinary axiom 6} is equivalent to the condition that, given a composite of the form $\bullet \stackrel{fib}{\longrightarrow} \bullet \stackrel{cof\ we}{\longrightarrow} \bullet$ which is a weak equivalence, the map marked $fib$ is also in $fib\ we$. These two axioms along with \eqref{fbpair trinary axiom 4} are equivalent to the weak equivalences satisfying the two-out-of-three axiom.
\end{itemize}
\end{proof}

\section{Explicit examples}
\label{Explicit examples}

\subsection{The power set of $\{a,b\}$}
\label{P of ab}

Any Boolean algebra with four elements is isomorphic to the power set $P(\{a,b\})$. The following diagram represents all objects and all nonidentity morphisms in $P(\{a,b\})$ when regarded as a category:
\begin{equation}\label{diag 02300}\xymatrix{
 & \{a,b\} & \\
\{a\} \ar[ur] & & \{b\} \ar[ul] \\
 & \emptyset \ar[ur]\ar[uu]\ar[ul] & 
}\end{equation}
\begin{notation}
    It will be convenient to use a miniature version of this diagram, with appropriate labellings, 
    to concisely describe factorization systems and model structures on the category $P(\{a,b\})$. 
    To concisely describe a factorization system $(\mathcal{L}, \mathcal{R})$ on $P(\{a,b\})$, 
    we will draw a miniature version of diagram \eqref{diag 02300}, 
    with an $L$ marked over each morphism which is in $\mathcal{L}$, 
    and an $R$ marked over each morphism which is in $\mathcal{R}$. For example, the drawing 
\tikzset{
 d5/.pic={
\draw[lightgray, very thin] (0,-1) -- (1,0) -- (0,1) -- (-1,0) -- cycle; 
\draw[lightgray, very thin] (0,-1) -- (0,1);
\path (-.5,0.5) node[color=black] {$R$} ;
\path (.5,0.5) node[color=black] {$L$} ;
\path (0,0) node[color=black] {} ;
\path (-.5,-0.5) node[color=black] {$L$} ;
\path (.5,-0.5) node[color=black] {$R$};
}}
\begin{tikzpicture} [shorten <= 0.4cm,shorten >= 0.4cm,
 d5/.style={path picture={\pic at ([yshift=-0.1cm]path picture bounding box) {d5};}},
every pic/.style={scale=.2,every node/.style={scale=0.4}},
] 
\draw (0,0) node[d5,minimum size=0.8cm,color=black] {};
\end{tikzpicture}
denotes the factorization system on $P(\{a,b\})$ in which $\emptyset\rightarrow \{a\}$ and $\{b\} \rightarrow \{a,b\}$ are in $\mathcal{L}$, while $\emptyset\rightarrow \{b\}$ and $\{a\} \rightarrow \{a,b\}$ are in $\mathcal{R}$, and $\emptyset \rightarrow \{a,b\}$ is in neither $\mathcal{L}$ nor $\mathcal{R}$.
\end{notation}

With this notational convention in place, the following figure shows all ten of the factorization systems on $P(\{a,b\})$ (see also Example 1.8 in \cite{leftandright}, where these ten factorization systems are also depicted):

\begin{equation}\label{n2 fact systems}
\tikzset{
 d0/.pic={
\draw[lightgray, very thin] (0,-1) -- (1,0) -- (0,1) -- (-1,0) -- cycle; 
\draw[lightgray, very thin] (0,-1) -- (0,1);
\path (-.5,0.5) node[color=black] {$R$} ;
\path (.5,0.5) node[color=black] {$R$} ;
\path (0,0) node[color=black] {$R$} ;
\path (-.5,-0.5) node[color=black] {$R$} ;
\path (.5,-0.5) node[color=black] {$R$};
\path (0,2) node [color=black] {$0_{rtmg}^{sm}$};
}}
\tikzset{
 d1/.pic={
\draw[lightgray, very thin] (0,-1) -- (1,0) -- (0,1) -- (-1,0) -- cycle; 
\draw[lightgray, very thin] (0,-1) -- (0,1);
\path (-.5,0.5) node[color=black] {$L$} ;
\path (.5,0.5) node[color=black] {$R$} ;
\path (0,0) node[color=black] {$R$} ;
\path (-.5,-0.5) node[color=black] {$R$} ;
\path (.5,-0.5) node[color=black] {$R$};
\path (0,2) node [color=black] {$1_{rt}$};
}}
\tikzset{
 d2/.pic={
\draw[lightgray, very thin] (0,-1) -- (1,0) -- (0,1) -- (-1,0) -- cycle; 
\draw[lightgray, very thin] (0,-1) -- (0,1);
\path (-.5,0.5) node[color=black] {$R$} ;
\path (.5,0.5) node[color=black] {$L$} ;
\path (0,0) node[color=black] {$R$} ;
\path (-.5,-0.5) node[color=black] {$R$} ;
\path (.5,-0.5) node[color=black] {$R$};
\path (0,2) node [color=black] {$2_{rt}$};
}}
\tikzset{
 d3/.pic={
\draw[lightgray, very thin] (0,-1) -- (1,0) -- (0,1) -- (-1,0) -- cycle; 
\draw[lightgray, very thin] (0,-1) -- (0,1);
\path (-.5,0.5) node[color=black] {$L$} ;
\path (.5,0.5) node[color=black] {$L$} ;
\path (0,0) node[color=black] {$R$} ;
\path (-.5,-0.5) node[color=black] {$R$} ;
\path (.5,-0.5) node[color=black] {$R$};
\path (0,2) node [color=black] {$3_{rtm}$};
}}
\tikzset{
 d4/.pic={
\draw[lightgray, very thin] (0,-1) -- (1,0) -- (0,1) -- (-1,0) -- cycle; 
\draw[lightgray, very thin] (0,-1) -- (0,1);
\path (-.5,0.5) node[color=black] {$L$} ;
\path (.5,0.5) node[color=black] {$L$} ;
\path (0,0) node[color=black] {$L$} ;
\path (-.5,-0.5) node[color=black] {$R$} ;
\path (.5,-0.5) node[color=black] {$R$};
\path (0,2) node [color=black] {$4_{}^{stm}$};
}}
\tikzset{
 d5/.pic={
\draw[lightgray, very thin] (0,-1) -- (1,0) -- (0,1) -- (-1,0) -- cycle; 
\draw[lightgray, very thin] (0,-1) -- (0,1);
\path (-.5,0.5) node[color=black] {$R$} ;
\path (.5,0.5) node[color=black] {$L$} ;
\path (0,0) node[color=black] {} ;
\path (-.5,-0.5) node[color=black] {$L$} ;
\path (.5,-0.5) node[color=black] {$R$};
\path (0,2) node [color=black] {$5_{rm}^{sm}$};
}}
\tikzset{
 d6/.pic={
\draw[lightgray, very thin] (0,-1) -- (1,0) -- (0,1) -- (-1,0) -- cycle; 
\draw[lightgray, very thin] (0,-1) -- (0,1);
\path (-.5,0.5) node[color=black] {$L$} ;
\path (.5,0.5) node[color=black] {$L$} ;
\path (0,0) node[color=black] {$L$} ;
\path (-.5,-0.5) node[color=black] {$L$} ;
\path (.5,-0.5) node[color=black] {$R$};
\path (0,2) node [color=black] {$6^{st}$};
}}
\tikzset{
 d7/.pic={
\draw[lightgray, very thin] (0,-1) -- (1,0) -- (0,1) -- (-1,0) -- cycle; 
\draw[lightgray, very thin] (0,-1) -- (0,1);
\path (-.5,0.5) node[color=black] {$L$} ;
\path (.5,0.5) node[color=black] {$R$} ;
\path (0,0) node[color=black] {} ;
\path (-.5,-0.5) node[color=black] {$R$} ;
\path (.5,-0.5) node[color=black] {$L$};
\path (0,2) node [color=black] {$7_{rm}^{sm}$};
}}
\tikzset{
 d8/.pic={
\draw[lightgray, very thin] (0,-1) -- (1,0) -- (0,1) -- (-1,0) -- cycle; 
\draw[lightgray, very thin] (0,-1) -- (0,1);
\path (-.5,0.5) node[color=black] {$L$} ;
\path (.5,0.5) node[color=black] {$L$} ;
\path (0,0) node[color=black] {$L$} ;
\path (-.5,-0.5) node[color=black] {$R$} ;
\path (.5,-0.5) node[color=black] {$L$};
\path (0,2) node [color=black] {$8_{}^{st}$};
}}
\tikzset{
 d9/.pic={
\draw[lightgray, very thin] (0,-1) -- (1,0) -- (0,1) -- (-1,0) -- cycle; 
\draw[lightgray, very thin] (0,-1) -- (0,1);
\path (-.5,0.5) node[color=black] {$L$} ;
\path (.5,0.5) node[color=black] {$L$} ;
\path (0,0) node[color=black] {$L$} ;
\path (-.5,-0.5) node[color=black] {$L$} ;
\path (.5,-0.5) node[color=black] {$L$};
\path (0,2) node [color=black] {$9_{rm}^{stmg}$};
}}
\begin{tikzpicture} [shorten <= 0.4cm,shorten >= 0.4cm,
 d0/.style={path picture={\pic at ([yshift=-0.1cm]path picture bounding box) {d0};}},
 d1/.style={path picture={\pic at ([yshift=-0.1cm]path picture bounding box) {d1};}},
 d2/.style={path picture={\pic at ([yshift=-0.1cm]path picture bounding box) {d2};}},
 d3/.style={path picture={\pic at ([yshift=-0.1cm]path picture bounding box) {d3};}},
 d4/.style={path picture={\pic at ([yshift=-0.1cm]path picture bounding box) {d4};}},
 d5/.style={path picture={\pic at ([yshift=-0.1cm]path picture bounding box) {d5};}},
 d6/.style={path picture={\pic at ([yshift=-0.1cm]path picture bounding box) {d6};}},
 d7/.style={path picture={\pic at ([yshift=-0.1cm]path picture bounding box) {d7};}},
 d8/.style={path picture={\pic at ([yshift=-0.1cm]path picture bounding box) {d8};}},
 d9/.style={path picture={\pic at ([yshift=-0.1cm]path picture bounding box) {d9};}},
every pic/.style={scale=.2,every node/.style={scale=0.4}},
] 
\draw (0,3) node[d0,minimum size=0.8cm,circle,draw,color=black,fill=red!10] {};
\draw (-1,2) node[d1,minimum size=0.8cm,circle,draw,color=black,fill=blue!10] {};
\draw (1,2) node[d2,minimum size=0.8cm,circle,draw,color=black,fill=blue!10] {};
\draw (0,1) node[d3,minimum size=0.8cm,circle,draw,color=black, shading=true, left color=blue!10, right color=orange!10] {};
\draw (0,-1) node[d4,minimum size=0.8cm,circle,draw,color=black] {};
\draw (1,0) node[d5,minimum size=0.8cm,circle,draw,color=black,fill=orange!10] {};
\draw (1,-2) node[d6,minimum size=0.8cm,circle,draw,color=black] {};
\draw (-1,0) node[d7,minimum size=0.8cm,circle,draw,color=black,fill=orange!10] {};
\draw (-1,-2) node[d8,minimum size=0.8cm,circle,draw,color=black] {};
\draw (0,-3) node[d9,minimum size=0.8cm,circle,draw,color=black,fill=orange!10] {};
\draw [->] (0,3) -- (-1,2);
\draw [->] (0,3) -- (1,2);
\draw [->] (-1,2) -- (0,1);
\draw [->] (1,2) -- (0,1);
\draw [->] (-1,2) -- (-1,0);
\draw [->] (1,2) -- (1,0);
\draw [->] (0,1) -- (0,-1);
\draw [->] (-1,0) -- (-1,-2);
\draw [->] (1,0) -- (1,-2);
\draw [->] (0,-1) -- (-1,-2);
\draw [->] (0,-1) -- (1,-2);
\draw [->] (-1,-2) -- (0,-3);
\draw [->] (1,-2) -- (0,-3);
\end{tikzpicture}
\end{equation}

Figure \eqref{n2 fact systems} was obtained by using Proposition \ref{fact brackets and fact systems and model structs} and programming a computer\footnote{Our computer code for this and all later calculations in this paper, written in Sage \cite{sagemath95}, is freely available at \url{https://asalch.wayne.edu/model_structures_1.tar}. Many of the calculations should also be possible using Balchin's ninfty software \cite{balchin2025ninftysoftwarepackagehomotopical}, although not all: e.g. ninfty would certainly not have any stock functionality for finding all the matroidal model structures, marked in orange circles in \eqref{n2 model structs} for the four-element Boolean algebra, since matroidal model structures are a new construction in this paper.} to carry out a brute force calculation of all factorization brackets on $P(\{a,b\})$. Each circle in figure \eqref{n2 fact systems} represents a single factorization bracket, i.e., a single factorization system. The number at the top of each circle is simply an index number, so that we can easily refer to factorization systems on $P(\{a,b\})$ later in this section, e.g. by writing ``factorization system $5$'' to mean the factorization system 
\tikzset{
 d5/.pic={
\draw[lightgray, very thin] (0,-1) -- (1,0) -- (0,1) -- (-1,0) -- cycle; 
\draw[lightgray, very thin] (0,-1) -- (0,1);
\path (-.5,0.5) node[color=black] {$R$} ;
\path (.5,0.5) node[color=black] {$L$} ;
\path (0,0) node[color=black] {} ;
\path (-.5,-0.5) node[color=black] {$L$} ;
\path (.5,-0.5) node[color=black] {$R$};
}}
\begin{tikzpicture} [shorten <= 0.4cm,shorten >= 0.4cm,
 d5/.style={path picture={\pic at ([yshift=-0.1cm]path picture bounding box) {d5};}},
every pic/.style={scale=.2,every node/.style={scale=0.4}},
] 
\draw (0,0) node[d5,minimum size=0.8cm,color=black] {};
\end{tikzpicture}. 

The subscript $r$ next to the number indicates that the factorization system is retractile, 
in the sense of Definition \ref{retractile and sectile classes}. 
Similarly, a subscript $t$ marks each of the factorization systems which are topological (Definition \ref{topological def}).
A subscript $m$ marks each factorization system arising from a matroid with underlying set $\{a,b\}$, 
and a $g$ marks each factorization system arising from a geometry with underlying set $\{a,b\}$, as in Definition \ref{def of matroidal and geometric}. Because we think it is a useful visualization, we have also incorporated that information in the form of the background colors of the circles:
\begin{itemize}
\item white background color indicates a factorization system which is not retractile,
\item blue background color indicates a factorization system which is topological (hence also retractile), i.e., arises from a topology on $\{a,b\}$,
\item orange background color indicates a factorization system which is matroidal (hence also retractile), i.e., arises from a matroid with underlying set $\{a,b\}$,
\item mixed blue and orange background color indicates a factorization system which is both topological and matroidal,
\item and red background color indicates a factorization system which is geometric (hence also matroidal) and also topological.
\end{itemize}

We use superscripts for the dual notions: a superscript $s$ marks each factorization system which is sectile, 
while superscripts $t,m,g$ mark each factorization system whose dual arises from a topology, a matroid, or a geometry, respectively, with underlying set $\{a,b\}$. 

An arrow from factorization system $(\mathcal{L},\mathcal{R})$ to factorization system $(\mathcal{L}^{\prime},\mathcal{R}^{\prime})$ means that the latter is a {\em localization} of the former, i.e., we have a containment $\mathcal{R}^{\prime}\subseteq \mathcal{R}$. Equivalently (using Proposition \ref{fact brackets and fact systems and model structs}), in terms of factorization brackets: a factorization bracket $\Psi$ is a localization of a factorization bracket $\Phi$ iff, for every pair $X,Y$ with $X\leq Y$, we have $\Phi(\Psi(X,Y),Y) \cong \Psi(X,Y)$. 

\begin{observation}
    It is not {\em a priori} clear, but instead a consequence of the calculation, 
    that every retractile factorization system on $\{a,b\}$ is either topological or matroidal (or both), 
    and that every geometric factorization system on $\{a,b\}$ is also topological.
\end{observation}

\begin{observation}
    There is a great deal of apparent symmetry in figure \eqref{n2 fact systems}. 
    In general, if we draw the quiver of factorization systems on the Boolean algebra $P(\{ a_1, \dots ,a_n\})$, 
    the symmetric group $\Sigma_n$ acts on the quiver by permuting the members of $P(\{ a_1, \dots ,a_n\})$, 
    but the cyclic group $C_2$ also acts on the quiver by the action on $P(\{ a_1, \dots ,a_n\})$ 
    which sends each subset of $\{a_1, \dots ,a_n\}$ to its complement. 
    These actions commute with each other. The former action preserves localizations, 
    and is responsible for the evident left-to-right symmetry in figure \eqref{n2 fact systems}. 
    On the other hand, the latter action reverses localizations, 
    and is responsible for the evident top-to-bottom symmetry in figure \eqref{n2 fact systems}.
\end{observation}

With figure \eqref{n2 fact systems} in hand, it is not difficult to find all model structures on $P(\{a,b\})$---or, for that matter, any finite partially-ordered set---by the following algorithm:
\begin{enumerate}
\item List all the pairs $(\Phi,\Psi)$ of factorization brackets on $P(\{a,b\})$ such that $\Psi$ is a localization of $\Phi$. Inspection of figure \eqref{n2 fact systems} reveals that there are precisely 
$44$ such pairs.
\item For each such pair $(\Phi,\Psi)$, check whether it satisfies the axioms in the definition of a factorization bracket pair in Definition \ref{def of fact bracket}. Proposition \ref{fact brackets and fact systems and model structs} then tells us that each such pair corresponds to a model structure on $P(\{a,b\})$, and vice versa.
\end{enumerate}
We carried out this process twice: once by computer, using the Sage script we wrote which is available at \url{https://asalch.wayne.edu/model_structures_1.tar}, and once entirely by hand, to be certain that we arrived at the same results. The algorithm used by our computer software\footnote{Our software implementing this algorithm takes only a few seconds on the power set of a two-element set. On the power set of a three-element set, it took about a day and a half on Wayne State University's high-performance computing grid. On the power set of a four-element set, this algorithm is too inefficient to run on any computer whatsoever. The difficulty is entirely in the first step. It is the most computationally costly step by an enormous margin.} is quite na\"{i}ve: 
\begin{enumerate}
\item Calculate the set of factorization brackets (Definition \ref{def of fact bracket}) on $P(\{a,b\})$, by looping over all possible functions $P(\{a,b\}) \rightarrow P(\{a,b\})\times P(\{a,b\})$, checking whether each given such function satisfies the axioms for a factorization bracket, and if is does, appending it to a list. This yields a list of all factorization brackets on $P(\{a,b\})$.
\item Make a list of ordered pairs of members of the list of factorization brackets on $P(\{a,b\})$, and then loop over that list of ordered pairs, checking whether each such ordered pair satisfies the axioms for a factorization bracket, and if it does, appending it to another list. This yields a list of all factorization bracket pairs on $P(\{a,b\})$, i.e., by Proposition \ref{fact brackets and fact systems and model structs}, a list of all model structures on $P(\{a,b\})$.
\item Loop over the list of factorization bracket pairs, checking whether each one satisfies the various properties of interest: strong, topological, matroidal, etc.
\item Make a list of ordered pairs of members of the list of factorization bracket pairs, and then for each pair on that list, check whether the second pair is a Bousfield localization of the first. Carry out the same check for Bousfield colocalizations.
\end{enumerate}
The result of the calculation is that there are precisely $23$ model structures on $P(\{a,b\})$, of which $17$ are strong, $9$ are topological, $11$ are matroidal, and one is geometric. They are as pictured earlier in the paper, in \eqref{n2 model structs}.

\subsection{The power set of $\{a\}$}
\label{P of a}

With the conventions established in the previous subsection, it is an easy matter to work out the quiver of factorization systems on $P(\{a\})$. There are only two factorization systems:
\begin{equation*}
\tikzset{
 d0/.pic={
\draw[lightgray, very thin] (0,-1) -- (0,1); 
\path (0,0) node[color=black] {$R$} ;
\path (0,2) node [color=black] {$0_{rtmg}^{sm}$};
}}
\tikzset{
 d1/.pic={
\draw[lightgray, very thin] (0,-1) -- (0,1); 
\path (0,0) node[color=black] {$L$} ;
\path (0,2) node [color=black] {$1_{rm}^{stmg}$};
}}
\begin{tikzpicture} [shorten <= 0.4cm,shorten >= 0.4cm,
 d0/.style={path picture={\pic at ([yshift=-0.1cm]path picture bounding box) {d0};}},
 d1/.style={path picture={\pic at ([yshift=-0.1cm]path picture bounding box) {d1};}},
every pic/.style={scale=.2,every node/.style={scale=0.4}},
] 
\draw (0,0) node[d0,minimum size=0.8cm,circle,draw,color=black,fill=red!10] {};
\draw (1,0) node[d1,minimum size=0.8cm,circle,draw,color=black,fill=orange!10] {};
\draw [->] (0,0) -- (1,0);
\end{tikzpicture}
\end{equation*}
The two factorization systems on $P(\{a\})$ are both retractile and sectile.
There are also only three model structures on $P(\{a\})$:
\begin{equation*}
\tikzset{
 d00/.pic={
\draw[lightgray, very thin] (0,-1) -- (0,1); 
\draw[green] (0,-1) node[circle,minimum size=0cm,fill=green!40] {};
\path (0,0) node[color=black] {$fw$} ;
\path (0,2) node [color=black] {0,0};
}}
\tikzset{
 d01/.pic={
\draw[lightgray, very thin] (0,-1) -- (0,1); 
\draw[green] (0,-1) node[circle,minimum size=0cm,fill=green!40] {};
\draw[green] (0,1) node[circle,minimum size=0cm,fill=green!40] {};
\path (0,0) node[color=black] {$cf$} ;
\path (0,2) node [color=black] {0,1};
}}
\tikzset{
 d11/.pic={
\draw[lightgray, very thin] (0,-1) -- (0,1); 
\draw[green] (0,1) node[circle,minimum size=0cm,fill=green!40] {};
\path (0,0) node[color=black] {$cw$} ;
\path (0,2) node [color=black] {1,1};
}}
\begin{tikzpicture} [shorten <= 0.4cm,shorten >= 0.4cm,
 d00/.style={path picture={\pic at ([yshift=-0.1cm]path picture bounding box) {d00};}},
 d01/.style={path picture={\pic at ([yshift=-0.1cm]path picture bounding box) {d01};}},
 d11/.style={path picture={\pic at ([yshift=-0.1cm]path picture bounding box) {d11};}},
every pic/.style={scale=.2,every node/.style={scale=0.4}},
] 
\draw ( 0,0) node[d01,minimum size=0.8cm,circle,draw,color=red,fill=red!10] {};
\draw ( 2,0) node[d11,minimum size=0.8cm,circle,draw,color=orange,fill=orange!10] {};
\draw (-2,0) node[d00,minimum size=0.8cm,circle,draw,color=orange,fill=orange!10] {};
\draw [->,dashed,color=black] ( 0,0) -- (-2,0);
\draw [->,color=black] ( 0,0) -- ( 2,0);
\end{tikzpicture}
\end{equation*}
All three model structures on $P(\{a\})$ are strong in the sense of Definition \ref{def of strong model structures}.

\subsection{The power set of $\{a,b,c\}$}
\label{P3}

Using the same algorithm described in \cref{P of ab}, we calculate that there are 450 factorization systems and 1026 model structures on $P(\{a,b,c\})$. In this subsection we go into some detail about the set of all such factorization systems and the set of all such model structures, since to our knowledge this is, to date, the most elaborate worked-out example of its kind, and we think that having some record of various structural features in such a nontrivial case is worth having in the literature. 

The factorization systems on $P(\{a,b,c\})$ are as follows:

\begin{equation}\label{n3factorizationsystemsquiver}
\includegraphics{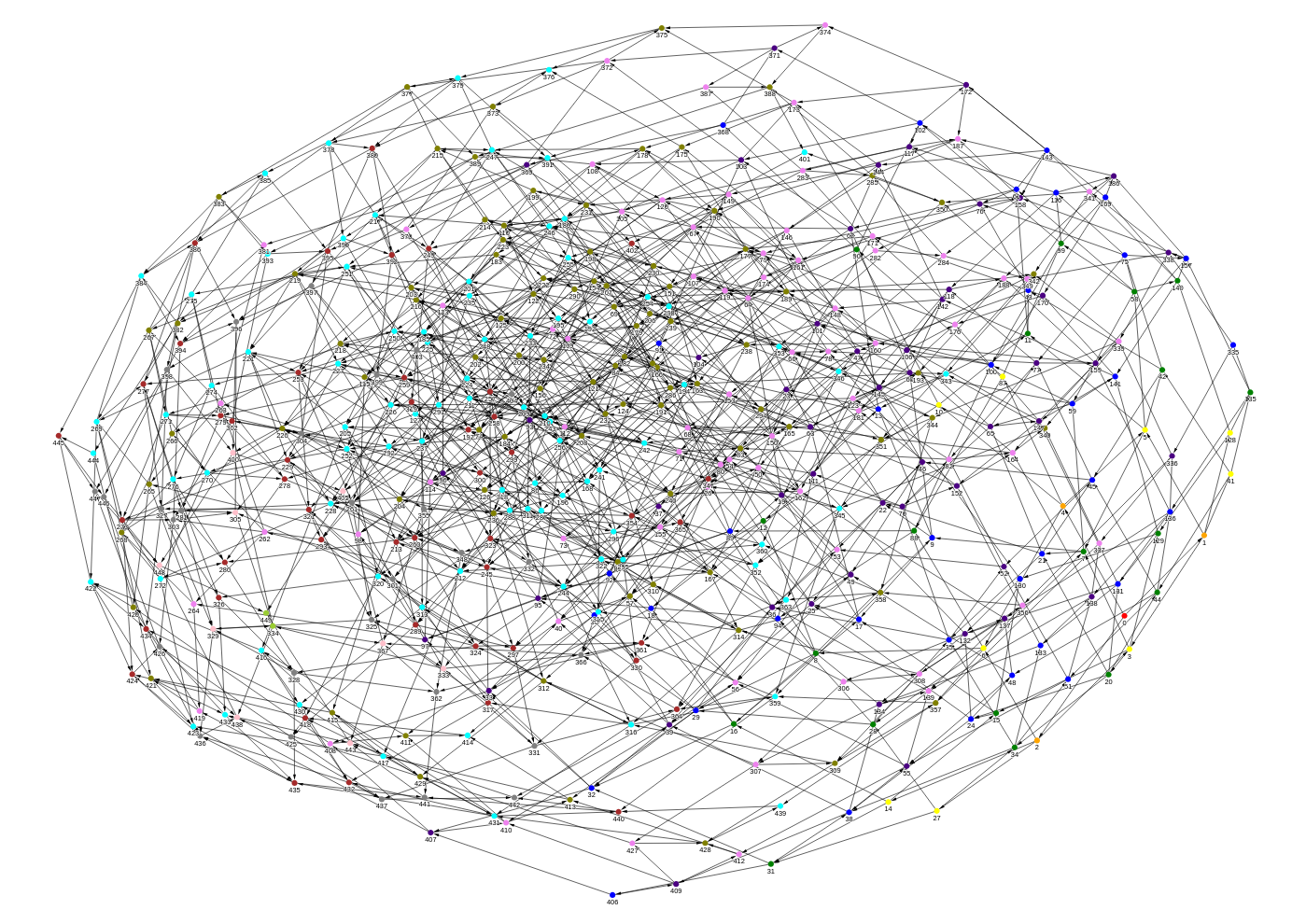}
\end{equation}

Due to limitations of page space, diagram \eqref{n3factorizationsystemsquiver} is unavoidably almost impossible to read. Hence, we describe verbally some of the important structure of diagram \eqref{n3factorizationsystemsquiver}. The factorization systems are colored by distance from the ``discrete'' factorization system, i.e., the factorization system $(\mathcal{L},\mathcal{R})$ in which $\mathcal{R}$ consists of all maps, and $\mathcal{L}$ consists of the isomorphisms. Let us say that a factorization system has {\em distance $n$ from the discrete model structure} if it is obtainable from the discrete model structure by a sequence of $n$ nontrivial localizations, but not by a sequence of $n+1$ nontrivial localizations.
Then the coloring is as follows:
\begin{itemize}
\item The discrete factorization system is colored red. 
\item The factorization systems of distance $1$ from the discrete model structure are colored orange. There are three of these. 
\item The factorization systems of distance $2$ from the discrete model structure are colored yellow. There are nine of these. 
\item This sequence proceeds until we reach distance $12$ from the discrete model structure, after which there are no more factorization systems. 
\end{itemize}
We tabulate the color codes, and the number of factorization systems of each color:

\begin{tabular}{|lll|}\hline
distance from disc. fact. sys. & color & number of fact. systems \\
\hline
0 & red & 1 \\
1 & orange & 3 \\
2 & yellow & 9 \\
3 & green & 19 \\
4 & blue & 36 \\
5 & indigo & 51\\
6 & violet & 70\\
7 & olive & 90\\
8 & cyan & 86\\
9 & brown & 46\\
10 & gray & 28\\
11 & pink & 9\\
12 & yellow-green & 2\\
\hline
\end{tabular}

Now for the model structures: of the 1026 model structures on $P(\{a,b,c\})$, 765 are in the same connected component as the discrete model structure. Those 765 model structures are as follows:

\begin{equation}\label{n3modelstructsquiver}
\includegraphics{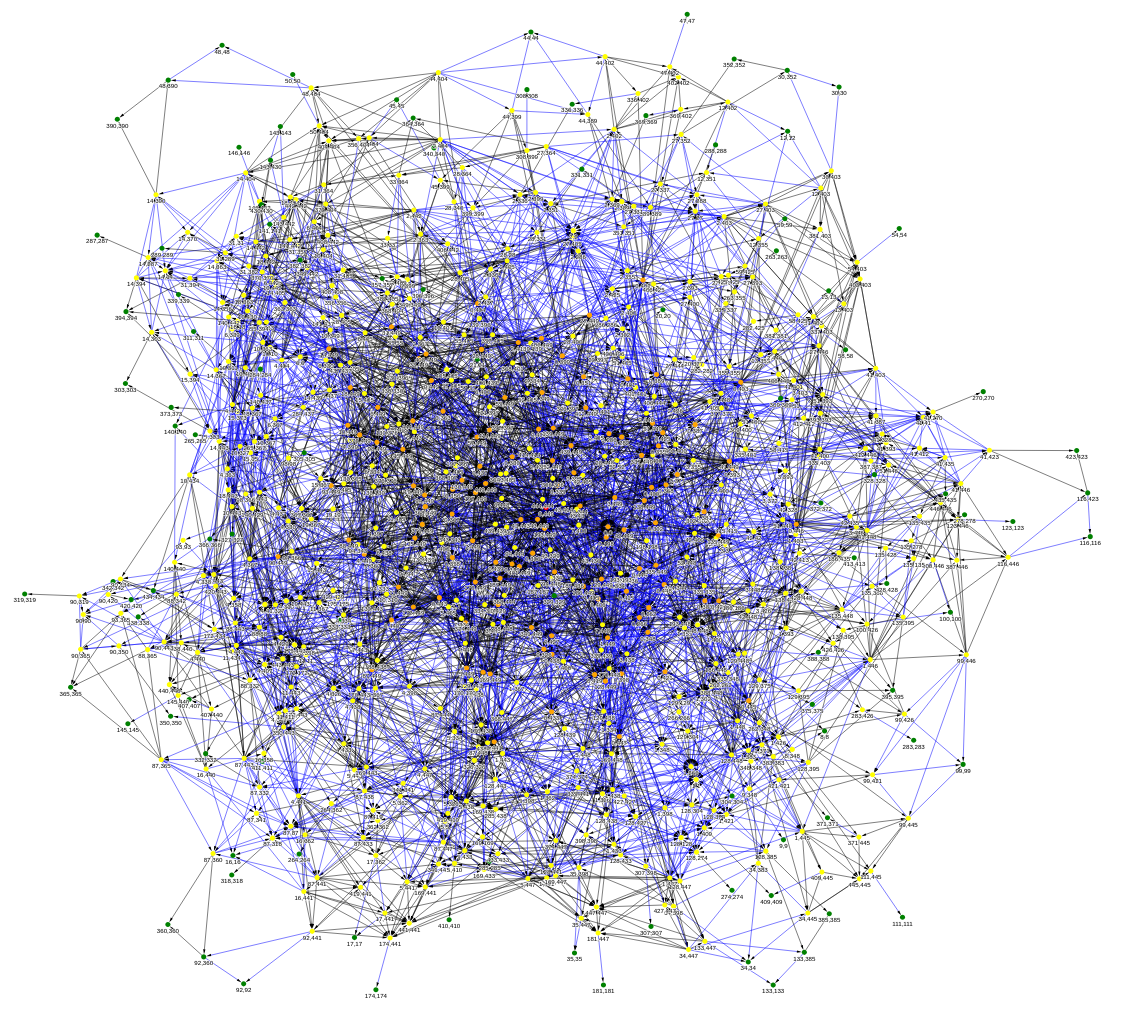}
\end{equation}

Again, since page space limitations unavoidably make diagram \eqref{n3modelstructsquiver} nearly unreadable, 
we give a verbal description of its most important features.
\begin{itemize}
    \item The discrete model structure is colored red.
    \item There are 120 model structures of distance 1, colored orange,
          from the discrete model structure.
    \item There are 542 model structures of distance 2, colored yellow, 
          from the discrete model structure.
    \item There are 102 model structures of distance 3, colored green, 
          from the discrete model structure.
\end{itemize}
To be clear, when we say that a model structure is ``of distance $n$ from the discrete model structure,'' we mean that one can arrive at that model structure from the discrete model structure by an alternating sequence, of length $n$, of Bousfield localizations and colocalizations. Not every model structure on $P(\{a,b\})$ has a well-defined distance from the discrete model structure, since not every model structure is obtainable from the discrete model structure by a sequence of Bousfield localizations and colocalizations. By {\em the main component,} we will mean the set of model structures with a well-defined distance from the discrete model structure. 

\begin{observation}
    It turns out that every model structure on $P(\{a,b\})$ which has a localization or colocalization in the main component is itself already in the main component. Perhaps there is some {\em a priori} argument one can make for this fact; we know it only as a consequence of calculating all 1026 model structures and the relationships of localization and colocalization between them. We also emphasize that the quiver of model structures on $P(\{a,b,c\})$ is not connected: one cannot get every model structure on $P(\{a,b,c\})$ by repeated localization and colocalization, starting from the discrete model structure. This stands in contrast to the situation for $P(\{a,b\})$ and $P(\{a\})$, where, as a consequence of the calculations in \cref{P of ab} and \cref{P of a}, we see that one {\em is} able to arrive at any model structure by repeated localization and colocalization, starting from the discrete model structure.  
\end{observation}

The rest of this subsection consists of further small observations about model structures on $P(\{a,b,c\})$, some of which hint at possible generalizations to model structures on larger Boolean algebras.

Aside from the 765 model structures in the main component, there are also 21 connected components of the form
\begin{equation}\label{component 2}
\begin{tikzpicture} [shorten <= 0.4cm,shorten >= 0.4cm,
every pic/.style={scale=.2,every node/.style={scale=0.4}},
] 
\draw ( 0,0) node[minimum size=0.8cm,circle,draw,color=black] {};
\draw ( 2,0) node[minimum size=0.8cm,circle,draw,color=black] {};
\draw (-2,0) node[minimum size=0.8cm,circle,draw,color=black] {};
\draw [->,dashed,color=black] ( 0,0) -- (-2,0);
\draw [->,color=black] ( 0,0) -- ( 2,0);
\end{tikzpicture}
\end{equation}
The remaining 198 model structures on $P(\{a,b\})$ are isolated, i.e., each of the remaining 198 model structures is not a localization or colocalization of any other model structure, and admits no nontrivial localizations or colocalizations.

Curiously, none of the isolated model structures are strong, and furthermore, neither are any of the model structures in the 21 connected components of the form \eqref{component 2}.
That is, every strong model structure on $P(\{a,b,c\})$ can be reached from the discrete model structure by iterated Bousfield localization and colocalization. 

Among the 765 model structures on $P(\{a,b,c\})$ in the main component, 
\begin{itemize}
\item 84 are topological (hence also strong),
\item 50 are matroidal (hence also strong)
\item 4 are geometric (hence also matroidal and strong), 
\item 252 are strong but not topological and not matroidal (hence also not geometric),
\item and 388 are not strong. 
\end{itemize}
These numbers do not add up to 765 because there is some overlap between them. 
In particular, 9 model structures are topological and matroidal. 
Precisely one of those---the discrete model structure---is also geometric. 

As observed above, all model structures on $P(\{a,b,c\})$ in the main component 
are of distance $\leq 3$ from the discrete model structure. 
Curiously, the strong model structures are all of distance $\leq 2$ from 
the discrete model structure, and the model structures which are both topological 
and matroidal are all of distance $\leq 1$ from the discrete model structure.

\def\cprime{$'$} \def\cprime{$'$} \def\cprime{$'$} \def\cprime{$'$}

\end{document}